\documentclass[11pt]{scrartcl}

\usepackage[utf8]{inputenc}

\usepackage{config}
\usepackage{titling}

\usepackage[a4paper, top=3cm, bottom=2cm, left=2.5cm, right=2.5cm, includefoot]{geometry}

\usepackage[inline]{enumitem}

\setlist[itemize]{topsep=0pt,partopsep=0pt,itemsep=0pt,parsep=0pt}
\setlist[itemize,1]{label={\small\textbullet}}
\setlist[itemize,2]{label={\tiny\textbullet}}
\setlist[itemize,3]{label=$\cdot$}
\setlist[enumerate]{topsep=0pt,partopsep=0pt,itemsep=0pt,parsep=0pt}
\setlist[enumerate,1]{label=\roman*)}
\setlist[enumerate,2]{label=\alph*)}
\setlist[enumerate,3]{label=\arabic*)}

\hypersetup{
	colorlinks=true,
	linkcolor=AO!65!black,
	citecolor=AO!65!black,
	urlcolor=AppleGreen!65!black,
	bookmarksopen=true,
	bookmarksnumbered,
	bookmarksopenlevel=2,
	bookmarksdepth=3
      }

\title{On non-planar, cycle-conformal graphs}
\predate{}
\date{}
\postdate{}

\preauthor{}
\DeclareRobustCommand{\authorthing}{
	\begin{center}
		Maximilian Gorsky\thanks{Supported by the Institute for Basic Science (IBS-R029-C1).}~~\!\footnote{\href{mailto:m.gorsky@pm.me}{m.gorsky@pm.me}}\\
		Discrete Mathematics Group, Institute for Basic Science (IBS), Daejeon, South Korea

        \medskip

            Clemens Kuske\footnote{\href{mailto:mail@clemens-kuske.de}{mail@clemens-kuske.de}}\\
		Technische Universität Berlin, Berlin, German
\end{center}}
\author{\authorthing}
\postauthor{}

\begin{document}
\maketitle

\begin{abstract}
A graph $G$ is called \emph{matching covered} if all of its edges are contained in some perfect matching of $G$.
Furthermore, a cycle $C \subseteq G$ is called \emph{conformal} if $G - V(C)$ has a perfect matching and $G$ itself is called \emph{cycle-conformal} if all of its even cycles are conformal.
Both matching covered graphs and conformal cycles play central roles in matching theory.

After a string of results from various authors, focused mainly on bipartite, planar graphs and claw-free graphs, a complete characterisation of all planar, cycle-conformal graphs has recently been presented by Dalwadi, Pause, Diwan, and Kothari [DMTCS, 2025].
We continue this exploration further into the realm of non-planar graphs, giving a characterisation of matching covered, cycle-conformal graphs that are bipartite and cubic, and respectively, those that are bipartite and Pfaffian.
The last class plays a fundamental role in matching theory, having important connections to the problem of counting perfect matchings, recognising graphs with even directed cycles, and computing the permanent of certain matrices efficiently.

To prove our results, we break matching covered graphs down to their building blocks, the bipartite ones of which are called \emph{braces}.
The key to both characterisations are theorems that identify the braces in the respective classes.
In particular, as our main results, we show that the cycle of length 4 is the only Pfaffian, cycle-conformal brace and we show that $K_{3,3}$ is the only cubic, cycle-conformal brace.
In both cases these theorems facilitate the characterisations of the much richer classes of associated matching covered graphs.
We conjecture that for each integer $\ell \geq 2$ the only $\ell$-regular, cycle-conformal brace is $K_{\ell,\ell}$.
\end{abstract}

\thispagestyle{empty}

\newpage

\setcounter{page}{1}

\section{Introduction}\label{sec:intro}
A core part of structural matching theory is the study of cycles whose deletion leaves us with a graph with a perfect matching.
Such cycles are called \emph{conformal}.
These cycles, and more generally subgraphs with this property, play a key role in structural matching theory~\cite{GiannopoulouW2021Two,GiannopoulouTW2023Excluding,GiannopoulouW2024Flat,GiannopoulouKW2024Excluding}, ear-decompositions (see~\cite{Szigeti1998Two,LovaszP2009Matching}), the resolution of Pólya's Permanent Problem\footnote{Pólya asked for the conditions under which, given a 0-1-matrix $A$, one could flip some of the 1s in $A$ to $-1$s to obtain the matrix $B$ such that the permanent of $A$ equals the determinant of $B$. This problem has numerous surprising connections to areas outside of linear algebra (see~\cite{RobertsonST1999Permanents,McCuaig2004Polyas}).}~\cite{RobertsonST1999Permanents,McCuaig2004Polyas}, and extensions of these results~\cite{FischerL2001Characterisation,NorineT2007Generating,GorskyKKW2024Packing,Gorsky2024Structure}.

Arising naturally from the significance of this property to the matching structure of graphs, one may therefore wonder which graphs have the property that \textsl{all} of their cycles are conformal.
However, any graph with a perfect matching that has odd cycles cannot have this property.
Thus, as it is natural in this context to study graphs with perfect matchings, we concentrate on the case in which we ask for all even cycles to be conformal.
We call a graph \emph{cycle-conformal} if it has a perfect matching and all of its even cycles are conformal.\footnote{This property has been studied under several names. We discuss our naming choice at the end of the introduction.}

Within the class of cycle-conformal graphs, we may further restrict ourselves to the study of \emph{matching covered} graphs, that is graphs with at least four vertices in which all edges are contained in a perfect matching of the graph.
These graphs are another central feature of matching theory.
We justify this restriction later on by showing that the characterisation of cycle-conformal graphs can easily be reduced to matching covered graphs (see \zcref{lem:reducetomatcov}).

The class of matching covered, cycle-conformal graphs first popped up explicitly when Guo and Zhang gave a characterisation of the planar, bipartite graphs in this class~\cite{GuoZ2004Reducible}, a result which has been replicated independently several times~\cite{ZhangL2012Computing,Pause2022Planar,Kuske2024Characterization} (see also~\cite{KlavzarS2012Characterization}).
Concerning matching covered, cycle-conformal, claw-free graphs, a characterisation of the planar graphs in this class was first given by Peng and Wang~\cite{PengW2019Ear} and a full characterisation was achieved by Zhang, Wang, Yuan, Ng, and Cheng~\cite{ZhangWYNC2022Cyclenice}.
Most recently Dalwadi, Pause, Diwan, and Kothari significantly advanced the topic by providing a full characterisation of planar, matching covered, cycle-conformal graphs~\cite{DalwadiPDK2025Planar}.

In general, the characterisations mentioned above are a bit too involved to state explicitly.
Once we restrict ourselves to the building blocks that make up matching covered graphs, there are however some notable patterns to observe.
Here two classes of graphs emerge from the \emph{tight cut decomposition procedure} due to Lov\'asz~\cite{Lovasz1987Matching}, introduced in detail later on in this introduction, that is of central importance in matching theory.
The graphs in the first class are called \emph{bricks}, which are nonbipartite, 3-connected graphs in which the deletion of any pair of vertices yields a graph with a perfect matching.
The second class consists of \emph{braces}, which are 2-connected, bipartite graphs in which any matching of size at most two can be extended to a perfect matching of the graph.
Thus the cycle of length 4, denoted as $C_4$, is a brace.

\paragraph{Our results.}
We first note the following result, which we deem to be folklore as it immediately follows from the characterisation in~\cite{GuoZ2004Reducible}, but is first mentioned explicitly in~\cite{DalwadiPDK2025Planar}.

\begin{proposition}[folklore]\label{thm:c4uniqueplanar}
    $C_4$ is the unique planar, cycle-conformal brace.\footnote{\cite{DalwadiPDK2025Planar} also consider $K_2$ to be a cycle-conformal brace due to a difference in their definition of what constitutes a brace. In fact, the only effect of this difference is whether or not $K_2$ is considered a brace.}
\end{proposition}

We will provide our own proof of this result later on (see \zcref{lem:planarbraces}).
The argument for this is fairly simple, making use of some very convenient topological properties of the planar setting (as do all other proofs of this that we are aware of).
Clearly, once we leave the realm of planar graphs, we can no longer directly rely on these methods.

Our first main contribution is a generalisation of this result to \emph{Pfaffian graphs}, whose definition we postpone to \zcref{sec:prelim}.
Bipartite Pfaffian graphs are the central class that needed to be characterised to resolve Pólya's Permanent Problem~\cite{RobertsonST1999Permanents,McCuaig2004Polyas} and they are notable for allowing the efficient computation of their number of perfect matchings, a problem which is otherwise known to be $\SharpP$\footnote{A complexity class for counting problems that is associated with the set of decision problems in $\NP$.}-complete~\cite{Valiant1979Complexity}.
Pfaffian graphs thus form an important fundamental class in matching theory.
Kasteleyn laid the groundwork for the study of these graphs by showing that all planar graphs are Pfaffian~\cite{Kasteleyn1963Dimer}.
Thus \zcref{thm:c4uniqueplanar} is properly generalised by our first main result.

\begin{theorem}\label{thm:c4uniquepfaffian}
    $C_4$ is the unique Pfaffian, cycle-conformal brace.
\end{theorem}

In our proof of this result, we use the characterisation of Pfaffian braces due to Robertson, Seymour, and Thomas~\cite{RobertsonST1999Permanents}, and McCuaig~\cite{McCuaig2004Polyas} to extend some convenient properties of planar graphs to Pfaffian graphs.

In light of \zcref{thm:c4uniquepfaffian}, one might suspect that there are only few cycle-conformal braces in total.
This is easily seen to be false however, as for all $t \in \mathbb{N}$ with $t \geq 2$, it is easy to prove that $K_{t,t}$ is a cycle-conformal brace and similarly, $K_{2t}$ is a cycle-conformal brick.
Furthermore, there even exist infinitely many planar, cubic, cycle-conformal bricks~\cite{DalwadiPDK2025Planar}.
This trend continues in non-planar, cubic, cycle-conformal bricks, as we remark upon in \zcref{lem:moebiscycconf}.
However, we provide the following result for cubic braces which starkly contrasts the situation for bricks.

\begin{theorem}\label{thm:k33issolo}
    The only cubic, cycle-conformal brace is $K_{3,3}$.
\end{theorem}

Unlike with \zcref{thm:c4uniquepfaffian}, the proof of this theorem is completely detached from the methods used for planar graphs.
Instead we first invest significant effort to show that we can find a special $K_{3,3}$-subdivision in any cubic, cycle-conformal brace and then show that there is no way to build any cycle-conformal graph around this special subgraph.

This remarkable difference in the richness of the class of cubic, cycle-conformal bricks and cubic, cycle-conformal braces motivated us to search for more cycle-conformal braces.
However, we could not find any aside from the aforementioned trivial examples furnished by the complete bipartite graphs.
In light of this we pose the following conjecture, for which \zcref{thm:k33issolo} resolves the case $\ell=3$ and $C_4 = K_{2,2}$ trivially verifies the case $\ell=2$.

\begin{conjecture}\label{con:main}
    For each integer $\ell \geq 2$, the only $\ell$-regular, cycle-conformal brace is $K_{\ell,\ell}$.
\end{conjecture}

In fact, we pose the following stronger conjecture, which would imply \zcref{con:main}.
Due to being considerably more general than \zcref{con:main}, our confidence that this second conjecture holds is notably lesser.

\begin{conjecture}\label{con:main2}
    Each cycle-conformal brace is a complete bipartite graph.
\end{conjecture}

Returning to the discussion of our results, we must mention that the restriction of our efforts to braces is not beyond reproach.
To explain this, we will now formally introduce the necessary definitions to outline the tight cut decomposition due to Lov\'asz.

In a graph $G$ with a set $X \subseteq V(G)$, we denote by $\partial(X) \subseteq E(G)$ the \emph{cut (around $X$)}, that is the collection of all edges incident with exactly one vertex in $X$.
By $\overline{X}$ we denote the set $V(G) \setminus X$.
If $G$ is matching covered, a cut $\partial(X)$ is called \emph{tight} if for all perfect matchings $M$ of $G$ we have $|M \cap \partial(X)| = 1$.
Such a tight cut is called \emph{trivial} if we have $|X| = 1$ or $|\overline{X}| = 1$.
Given a tight cut $\partial(X)$ in $G$, we denote by $G[X \mapsto c]$ the graph that results from contracting $X$ into a single vertex $c$ and deleting any resulting loops or parallel edges.
Analogously, $G[\overline{X} \mapsto c]$ is the graph that results from contracting $\overline{X}$ instead.
Both $G[X \mapsto c]$ and $G[\overline{X} \mapsto c]$ are called \emph{tight cut contractions} of $G$.
As Lov\'asz proved~\cite{Lovasz1987Matching}, the result of repeatedly taking tight cut contractions starting with a matching covered graph $G$ is a \textsl{unique} list of graphs that do not have non-trivial tight cuts, independent of which particular tight cuts were chosen during the procedure.
In particular, it is possible to carry out this process in polynomial time~\cite{Lovasz1987Matching}.
If a matching covered graph does not have non-trivial tight cuts it is called a \emph{brace}, if it is bipartite, and a \emph{brick}, if it is non-bipartite, which gives equivalent definitions of the two classes we introduced earlier~\cite{Plummer1986Matching,Lovasz1987Matching}.

While many problems in matching theory can be safely reduced to bricks and braces -- as several properties are preserved by the operation of taking tight cut contractions (see for example~\cite{VaziraniY1989Pfaffian,GorskyKKW2024Packing,Gorsky2024Structure} and \zcref{lem:tightcutpreservematcov}) -- this is not entirely true for the property of being cycle-conformal.
However, one of the two directions does hold.

\begin{proposition}[Dalwadi, Pause, Diwan, and Kothari~\cite{DalwadiPDK2025Planar}]\label{thm:tightcutpreserveeasy}
    Let $G$ be a matching covered graph with a tight cut $\partial(X)$ and let $G_1,G_2$ be the two tight cut contractions corresponding to $\partial(X)$.
    Then if $G$ is cycle-conformal, so are $G_1$ and $G_2$.
\end{proposition}

That a matching covered, cycle-conformal graph may have non-cycle-conformal bricks or braces was already pointed out in~\cite{DalwadiPDK2025Planar}.
This negative result can be extended, as such graphs exist even if all involved graphs are bipartite and planar (see \zcref{fig:tightcutcounterexample}).
Contrasting this, we show that cubic, matching covered, bipartite graphs are indeed closed under taking tight cut contractions.

\begin{lemma}\label{thm:cubictightcuts}
    A cubic, matching covered, bipartite graph is cycle-conformal if and only if its braces are cycle-conformal.
\end{lemma}

\begin{figure}[ht]
     \centering
        \begin{tikzpicture}
            \node (A0) at (1,2.25) [draw, circle, scale=0.6] {};
            \node (A1) at (1,1) [draw, circle, scale=0.6] {};
            \node (A2) at (2.5,1.75) [draw, circle, scale=0.6] {};
            \node (A3) at (3.25,0.5) [draw, circle, scale=0.6] {};

            \node (B0) at (1,1.5) [draw, circle, scale=0.6, fill] {};
            \node (B1) at (1,0.5) [draw, circle, scale=0.6, fill] {};
            \node (B2) at (1.75,1.5) [draw, circle, scale=0.6, fill] {};
            \node (B3) at (3.25,2.25) [draw, circle, scale=0.6, fill] {};

            \path
                (A0) edge[very thick] (B0)
                (A0) edge[very thick, bend right] (B1)
                (A0) edge[very thick] (B2)
                (A0) edge[very thick] (B3)
                (A1) edge[very thick] (B0)
                (A1) edge[very thick] (B1)
                (A1) edge[very thick] (B2)
                (A1) edge[very thick, bend right] (B3)
                (A2) edge[very thick] (B2)
                (A2) edge[very thick] (B3)
                (A3) edge[very thick] (B1)
                (A3) edge[very thick] (B3)
                ;
        \end{tikzpicture}
    \qquad
        \begin{tikzpicture}
            \node (A0) at (1,2.25) [draw, circle, scale=0.6] {};
            \node (A1) at (1,1) [draw, circle, scale=0.6] {};
            \node (A2) at (2.5,1.75) [draw, circle, scale=0.6] {};
            \node (A3) at (3.25,0.5) [draw, circle, scale=0.6] {};

            \node (B0) at (1,1.5) [draw, circle, scale=0.6, fill] {};
            \node (B1) at (1,0.5) [draw, circle, scale=0.6, fill] {};
            \node (B2) at (1.75,1.5) [draw, circle, scale=0.6, fill] {};
            \node (B3) at (3.25,2.25) [draw, circle, scale=0.6, fill] {};

            \path
                (A0) edge[very thick] (B0)
                (A0) edge[very thick, bend right,red] (B1)
                (A0) edge[very thick] (B2)
                (A0) edge[very thick,red] (B3)
                (A1) edge[very thick] (B0)
                (A1) edge[very thick,red] (B1)
                (A1) edge[very thick,red] (B2)
                (A1) edge[very thick, bend right] (B3)
                (A2) edge[very thick,red] (B2)
                (A2) edge[very thick,red] (B3)
                (A3) edge[very thick] (B1)
                (A3) edge[very thick] (B3)
                ;
        \end{tikzpicture}
    \qquad
	\begin{tikzpicture}
            \node (A0) at (1,2.25) [draw, circle, scale=0.6] {};
            \node (A1) at (1,1) [draw, circle, scale=0.6] {};
            \node (A2) at (2.5,1.75) [draw, circle, scale=0.6] {};
            \node (A3) at (3.25,0.5) [draw, circle, scale=0.6] {};

            \node (B0) at (1,1.5) [draw, circle, scale=0.6, fill] {};
            \node (B1) at (1,0.5) [draw, circle, scale=0.6, fill] {};
            \node (B2) at (1.75,1.5) [draw, circle, scale=0.6, fill] {};
            \node (B3) at (3.25,2.25) [draw, circle, scale=0.6, fill] {};

            \path
                (A0) edge[very thick] (B0)
                (A0) edge[very thick, bend right,dashed] (B1)
                (A0) edge[very thick,dashed] (B2)
                (A0) edge[very thick,dashed] (B3)
                (A1) edge[very thick] (B0)
                (A1) edge[very thick,dashed] (B1)
                (A1) edge[very thick,dashed] (B2)
                (A1) edge[very thick, bend right,dashed] (B3)
                (A2) edge[very thick] (B2)
                (A2) edge[very thick] (B3)
                (A3) edge[very thick] (B1)
                (A3) edge[very thick] (B3)
                ;
        \end{tikzpicture}
     \caption{A matching covered, planar, bipartite graph whose braces are all isomorphic to $C_4$. In the graph in the middle a non-conformal even cycle is indicated in red. On the right a tight cut is indicated such that both tight cut contractions corresponding to it are cycle-conformal.}
     \label{fig:tightcutcounterexample}
\end{figure}
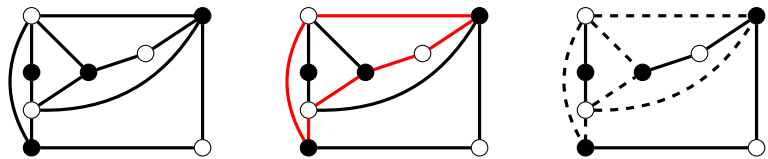

The counterexample from~\cite{DalwadiPDK2025Planar} we discussed above easily demonstrates that the condition of being bipartite is necessary in the above statement.
Combined with \zcref{thm:k33issolo}, the above lemma allows us to provide the following characterisation.

\begin{theorem}\label{thm:cubiccharacterisation}
    A cubic, bipartite, matching covered graph is cycle-conformal if and only if all of its braces are $K_{3,3}$.
\end{theorem}

Furthermore, the fact that tight cut contractions do not preserve being cycle-conformal tells us that \zcref{thm:c4uniquepfaffian} is not enough to itself characterise the Pfaffian, bipartite, cycle-conformal graphs.
However, with a bit of work this can be overcome, allowing us to extend \zcref{thm:c4uniquepfaffian} to a characterisation of matching covered, Pfaffian, bipartite, cycle-conformal graphs.

\begin{theorem}\label{thm:pfaffcycconfcharacterisation}
    Any matching covered, bipartite, cycle-conformal graph is Pfaffian if and only if it is planar.
\end{theorem}

For the convenience of the reader, we provide one of the characterisations of matching covered, bipartite, cycle-conformal, planar graphs at the end of \zcref{sec:tightcuts} (see \zcref{cor:Pfaffiancharacterisation} in particular).

\paragraph{Related work.}
One may also be interested in the number of perfect matchings that remain after the deletion of an even cycle.
If one demands that after the deletion of any even cycle the resulting graph has at most one perfect matching, this gives a definition for the class known as \emph{perfect matching compact} graphs (as pointed out in~\cite{WangZZ2018Characterization}).
These graphs have been studied extensively over the years~\cite{WangLCLSL2013Characterization,LiuW2014Note,WangSLC2014Characterization,WangYL2015Characterization,DeCarvalhoKWL2020Birkhoffvon,ZhouLFD2022Note,ZhangWY2022PMcompact}.
The problem of recognising perfect matching compact graphs is known to be in co-$\NP$ but not known to be in $\NP$ (see~\cite{DeCarvalhoKWL2020Birkhoffvon}), a situation which mirrors the state of the art for cycle-conformal graphs (see~\cite{DalwadiPDK2025Planar}).
Recently graphs in which the deletion of any cycle results in a unique perfect matching have started to be considered as well~\cite{WangZZ2018Characterization,ZhangWY2020Even,ZhangW2022Characterizations} and clearly any such graph is also cycle-conformal.
A specialised version of this is found in the study of \emph{forcing faces} in bipartite planar graphs (see~\cite{CheC2008Forcing,CheC2013Forcing}).

\paragraph{A note on terminology.}
There exist several other names in the literature for \emph{cycle-conformal} graphs.
They are also known as \emph{cycle-nice} (see~\cite{ZhangWYNC2022Cyclenice}), \emph{1-cycle-resonant} (see~\cite{KlavzarS2012Characterization}), or \emph{cycle-extendable} (see~\cite{DalwadiPDK2025Planar}) graphs.
Our choice to introduce yet another name for this class is grounded in an attempt to give this class a non-generic, appropriate, and distinct name.
Terms like \emph{good} and \emph{nice} are often used to generically describe a property of a graph in a variety of unrelated context and are thus not specific enough.
The property of being \emph{1-cycle-resonant} only coincides with being \emph{cycle-conformal} in the context of bipartite graphs, making it inappropriate for the class as a whole.
Finally, the term cycle-extendable does give a great description of what we expect our cycles to do,\footnote{Any perfect matching of an even cycle in a cycle-conformal graph $G$ can be extended to a perfect matching of $G$.} but there already exists a well-studied class of graphs that are called \emph{cycle extendable} (see the line of research based on the concepts in~\cite{Hendry1990Extending}).
We believe that the name \emph{cycle-conformal} is non-generic and accurately reflects that we want each (even) cycle to be conformal.
Furthermore, a search for the term ``cycle-conformal'' did not yield any other related graph-theoretic literature in conflict with this class.

We also note that the term \emph{conformal} itself has appeared in different guises, having been called \textsl{well-fitted}~\cite{McCuaig2004Polyas}, \textsl{nice}~\cite{Lovasz1987Matching}, and \textsl{central}~\cite{RobertsonST1999Permanents}.
Here we base our decision on the fact that the term \emph{conformal} has been the de facto consensus established throughout the most prominent publications in matching theory over the last two decades (see~\cite{Wiederrecht2022Matching,LucchesiM2024Perfect} for reference).

\paragraph{Structure of the paper.}
We start by providing the preliminaries necessary for our results in \zcref{sec:prelim}.
This is followed by the proof of \zcref{thm:c4uniquepfaffian} in \zcref{sec:braces} and a study of tight cuts in \zcref{sec:tightcuts}, which includes the proofs of \zcref{thm:cubictightcuts} and \zcref{thm:pfaffcycconfcharacterisation}.
In \zcref{sec:cubic} we provide the proof of \zcref{thm:k33issolo} from which \zcref{thm:cubiccharacterisation} then follows with ease.
We conclude our paper with a discussion of non-bipartite cycle-conformal graphs and possible directions for future research in \zcref{sec:generation}.

\section{Preliminaries}\label{sec:prelim}
The graphs in this paper are loopless and do not possess multiple edges.
We generally adopt our notation for graphs from~\cite{Diestel2010Graph}, except where otherwise noted.
For example, we denote the complete graph on $t$ vertices as $K_t$.

We will consider bipartite graphs throughout much of this article.
A graph $G$ is called \emph{bipartite} if its vertex set can be partitioned into two disjoint sets $A,B$ such that all edges of $G$ have one end in $A$ and the other in $B$.
We call $A$ and $B$ the \emph{colour classes} of $G$ and will often refer to $G$ having \emph{black} and \emph{white} vertices, which refers to containment of the vertices in these two sets.
These colours will also be used to indicate the colour classes of a bipartite graph in our figures.

We denote the complete bipartite graphs with $s$ white vertices and $t$ black vertices as $K_{s,t}$.
If $t \geq 3$, we let $C_t$ denote the cycle on $t$ vertices.
We will often call a cycle of length 4 a \emph{4-cycle}.

\paragraph{Paths, Separators, and Subdivisions.}
Let $G$ be a graph and let $k$ be a positive integer.
We call a set $S \subseteq V(G)$ a \emph{separator (in $G$)} if $G - S$ is non-empty and not connected.
If $S = \{ v \}$, we also call $v$ a \emph{cut-vertex} and in general, if $|S| = k$, we call $S$ a \emph{$k$-separator (in $G$)}.
The graph $G$ is called \emph{$k$-connected} if it contains at least $k+1$ vertices and there does not exist a $k'$-separator in $G$ with $k' < k$.

We call the non-endpoint vertices of a path its \emph{internal vertices}.
Should a path consist of a single vertex, we call it \emph{trivial}.
If $P$ is a path in a graph $G$ and $S \subseteq V(G)$, we say that $P$ is \emph{internally disjoint from $S$} if no internal vertex of $P$ lies in $S$.
More generally, if $H \subseteq G$, we say that $P$ is internally disjoint from $H$ if $P$ is internally disjoint from $V(H)$ and if $H$ is also a path, we say that $P$ and $H$ are internally disjoint if they are pairwise internally disjoint.

Let $A,B$ be two non-empty sets of vertices in a graph $G$, then a path $P$ with one endpoint in $A$ and the other in $B$ that is otherwise disjoint from $A \cup B$ is called an \emph{$A$-$B$-path}.
If $A = \{ u \}$ and $B = \{ v \}$, we also call $P$ \emph{a $u$-$v$-path}.
Let $a,b \in V(P)$ for a path $P$, then we denote by $aPb$ the unique $a$-$b$-path in $P$.
The following classic result connects the notions of $k$-connectivity and $A$-$B$-paths.

\begin{proposition}[Menger~\cite{Menger1927Zur}]\label{thm:menger}
    Let $k$ be a positive integer and let $G$ be a graph.
    Then $G$ is $k$-connected if and only if there exist $k$ pairwise disjoint $A$-$B$-paths in $G$.
\end{proposition}

Given a graph $G$ and an edge $e \in E(G)$, we say that $G'$ is constructed by \emph{subdividing $e$} if $G'$ is derived from $G$ by replacing $e$ with a path of length two, with the internal vertex of this path being the \emph{subdivision vertex}.
We say that $H$ is a \emph{subdivision of $G$} if $H$ is derived from $G$ by repeatedly subdividing edges.
In a related notion, we say that $G'$ is a \emph{bisubdivision of $G$}, if we can construct $G'$ from $G$ by replacing the edges of $G$ with pairwise internally disjoint odd paths.
Note that since the path with two vertices is odd, $G$ is a bisubdivision of itself.

\paragraph{Planar graphs.}
A graph $G$ is called \emph{planar} if there exists a drawing $D$ of it into the plane such that no two edges of $G$ intersect each other in anything but potentially their endpoints.
We call such a drawing $D$ a \emph{plane drawing (of $G$)}.
When deleting $D$ from the plane, the resulting connected regions are called the \emph{faces} of $G$.
Whitney proved that 3-connected, planar graphs have essentially unique drawings~\cite{Whitney19332Isomorphic}.
Thus as long as we are dealing with 3-connected, planar graphs we can tacitly assume that they are directly associated with a drawing.
Furthermore, for any 2-connected, planar graph $G$, Whitney showed that vertices and edges found on the boundary of the closure of any given face of $G$ correspond to a cycle in $G$~\cite{Whitney1932Nonseparable}.
We call the cycles corresponding to faces in a 2-connected, planar graph its \emph{facial cycles}.

Famously, planar graphs can be characterised by excluding subdivisions of two special graphs.

\begin{proposition}[Kuratowski~\cite{Kuratowski1930Probleme}]\label{thm:kuratowski}
    A graph is planar if and only if it does not contain a subdivision of $K_5$ or $K_{3,3}$.
\end{proposition}

\paragraph{Matching theory.}
For any positive $k \in \mathbb{N}$, a graph $G$ is called \emph{$k$-extendable} if it has at least $2k+2$ vertices and for any matching $N$ of order $k$ there exists a perfect matching $M$ of $G$ with $N \subseteq M$.
The following useful result stems from the article introducing $k$-extendability.

\begin{proposition}[Plummer~\cite{Plummer1980$n$extendable}]\label{thm:extendabilitybasics}
    For all positive integers $k$, any $k$-extendable graph is $(k-1)$-extendable and $(k+1)$-connected.
\end{proposition}

When considering braces, their characterisation via $k$-extendability offers a helpful alternative to having to argue via the absence of tight cuts, which we will be applying implicitly throughout this article.
The following is a consequence of results from~\cite{Plummer1986Matching} and~\cite{Lovasz1987Matching}.

\begin{proposition}
    A graph is a brace if and only if it is $C_4$, or it is bipartite and 2-extendable.
\end{proposition}

Regarding tight cuts, we want to briefly mention the following two useful folklore results.

\begin{lemma}[folklore]\label{lem:tightcutpreservematcov}
    Let $G$ be a matching covered graph with a tight cut $\partial(X)$.
    Then the two tight cut contractions corresponding to $\partial(X)$ are matching covered.
    If $G$ is also bipartite then so are the two tight cut contractions corresponding to $\partial(X)$.
\end{lemma}

\begin{lemma}[folklore]\label{lem:findbrickbrace}
    Let $G$ be a matching covered graph.
    Then there exists a tight cut in $G$ such that one of the two tight cut contractions corresponding to it is a brick or brace of $G$.
\end{lemma}

Next, we properly introduce a key notion mentioned in the introduction.
A graph with a perfect matching is called \emph{Pfaffian} if there exists an orientation of its edges such that each conformal, even cycle has an odd number of edges oriented in the same direction in either direction of traversal.
The first major result concerning Pfaffian graphs is due to Kasteleyn.

\begin{proposition}[Kasteleyn~\cite{Kasteleyn1963Dimer}]\label{thm:planarpfaff}
    All planar graphs with perfect matchings are Pfaffian.
\end{proposition}

However, the result of most interest to us is the following characterisation of Pfaffian braces.
Let $G_1, G_2, G_3$ be three bipartite graphs, such that their pairwise intersection is a 4-cycle $C$, and we have $V(G_i) \setminus V(C) \neq \emptyset$ for all $i \in [3]$.
Further, let $S \subseteq E(C)$ be some subset of the edges of $C$.
The \emph{trisum of $G_1, G_2, G_3$ at $C$} is a graph $\bigcup_{i=1}^3 G_i - S$.

\begin{proposition}[McCuaig~\cite{McCuaig2004Polyas}, Robertson, Seymour, and Thomas~\cite{RobertsonST1999Permanents}]\label{thm:pfaffian}
    A brace is Pfaffian if and only if it is the Heawood graph (see \zcref{fig:heawood}) or it can be constructed from planar braces by repeated use of the trisum operation.
\end{proposition}

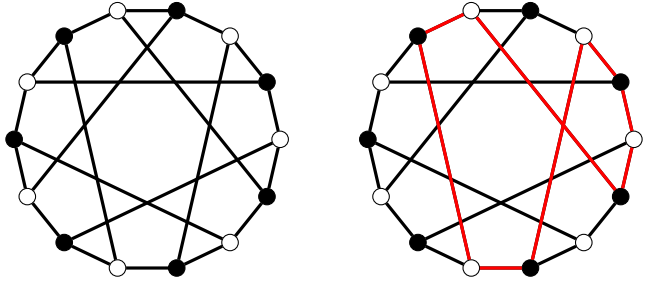
\begin{figure}[ht]
     \centering
            \begin{tikzpicture}[scale=1]
			
			\node (V0) at (0:0) [draw=none] {};
			
			\foreach\i in {1,3,5,7,9,11,13}
			{
				\node (V\i) at ($(V0)+({(360/14 * \i)}:1.75)$) [draw, circle, scale=0.6, fill, label={}] {};
			}
			
			\foreach\i in {2,4,6,8,10,12,14}
			{
				\node (V\i) at ($(V0)+({(360/14 * \i)}:1.75)$) [draw, circle, scale=0.6, label={}] {};
			}
			
			\foreach\i in {1,2,3,4,5,6,7,8,9,10,11,12,13}
			{
				\pgfmathtruncatemacro\iplus{\i+1}
				\path (V\i) edge[very thick] (V\iplus);
			}
			\path (V14) edge[very thick] (V1);
			
			\path
				(V1) edge[very thick] (V6)
				(V2) edge[very thick] (V11)
				(V3) edge[very thick] (V8)
				(V4) edge[very thick] (V13)
				(V5) edge[very thick] (V10)
				(V7) edge[very thick] (V12)
				(V9) edge[very thick] (V14)
			;
			
            \end{tikzpicture}
     \qquad
            \begin{tikzpicture}[scale=1]
			
			\node (V0) at (0:0) [draw=none] {};
			
			\foreach\i in {1,3,5,7,9,11,13}
			{
				\node (V\i) at ($(V0)+({(360/14 * \i)}:1.75)$) [draw, circle, scale=0.6, fill, label={}] {};
			}
			
			\foreach\i in {2,4,6,8,10,12,14}
			{
				\node (V\i) at ($(V0)+({(360/14 * \i)}:1.75)$) [draw, circle, scale=0.6, label={}] {};
			}
			
			\foreach\i in {1,2,3,4,5,6,7,8,9,10,11,12,13}
			{
				\pgfmathtruncatemacro\iplus{\i+1}
				\path (V\i) edge[very thick] (V\iplus);
			}
			\path (V14) edge[very thick] (V1);
			
			\path
				(V1) edge[very thick] (V6)
				(V2) edge[very thick] (V11)
				(V3) edge[very thick] (V8)
				(V4) edge[very thick] (V13)
				(V5) edge[very thick] (V10)
				(V7) edge[very thick] (V12)
				(V9) edge[very thick] (V14)
			;

                \path[red]
				(V2) edge[very thick] (V11)
                    (V11) edge[very thick] (V10)
                    (V10) edge[very thick] (V5)
                    (V5) edge[very thick] (V4)
                    (V4) edge[very thick] (V13)
                    (V13) edge[very thick] (V14)
                    (V14) edge[very thick] (V1)
                    (V1) edge[very thick] (V2)
			;
		\end{tikzpicture}
     \caption{A drawing of two copies of the Heawood graph. On the right a non-conformal, even cycle in the Heawood graph is marked in red.}
     \label{fig:heawood}
\end{figure}

Finally, we briefly discuss matching covered graphs and their role in the characterisation of cycle-conformal graphs.
Note that according to our definition a graph is matching covered if and only if it is 1-extendable.
It is sometimes useful to think of the \emph{cover graph} $\Cov{G}$ of $G = (V(G),E(G))$ which is defined such that $V(\Cov{G}) = V(G)$ and $E(\Cov{G})$ consists of all edges in $E(G)$ that are found in a perfect matching of $G$.

The cover graph of a graph can be computed in $\mathcal{O}(mn)$-time thanks to an algorithm by de Carvalho and Cheriyan~\cite{DeCarvalhoC2005An}, which in particular allows us to check whether a graph is matching covered in $\mathcal{O}(mn)$-time.
This allows us to show that efficient recognition of cycle-conformal graphs can be reduced to matching covered graphs.

\begin{lemma}\label{lem:reducetomatcov}
    Let $G$ be a graph with a perfect matching and let $H_1, \ldots , H_k$ be the components of its cover graph.
    Then $G$ is cycle-conformal if and only if
    \begin{enumerate}
        \item $H_1, \ldots , H_k$ are cycle-conformal, and

        \item no edge in $F \coloneqq E(G) \setminus \bigcup_{i=1}^k E(H_i)$ is contained in an even cycle of $G$.
    \end{enumerate}
    In particular, if we know that $H_1, \ldots , H_k$ are cycle-conformal, we can recognise whether $G$ is cycle-conformal in $\mathcal{O}(m^2)$-time.
\end{lemma}
\begin{proof}
    We first confirm the forward direction, assuming that $G$ is cycle-conformal.
    Suppose the first condition does not hold and w.l.o.g.\ $H_1$ is not cycle-conformal.
    Then there exists a cycle $C \subseteq H_1$ such that $H_1 - V(C)$ does not have a perfect matching, but $G - V(C)$ does have a perfect matching $M$.
    Note now that by definition of a cover graph all edges covering vertices in $V(H_1)$ must be contained in $H_1$ and thus, we have a contradiction to $H_1 - V(C)$ not having a perfect matching.

    Next, consider the possibility of the existence of an even cycle $C \subseteq G$ such that $E(C) \cap F \neq \emptyset$ and let $M$ be a perfect matching of $G - V(C)$.
    Then $C$ has two perfect matchings $M_1,M_2$, at least one of which, say $M_1$, contains an edge of $F$ and thus $M \cup M_1$ is a perfect matching of $G$, contradicting the definition of $F$.

    Thus, we must now only prove the backwards direction.
    For this purpose let $C \subseteq G$ be an even cycle and note that, according to the second condition, we have $E(C) \cap F = \emptyset$ and thus $C \subseteq H_i$, for some $i \in [k]$.
    Let $M_i$ be the perfect matching of $H_i - V(C)$ that must exist since $H_i$ is cycle-conformal and let $M_j$ be a perfect matching for each $i \in [k] \setminus \{ i \}$.
    Then $\bigcup_{i=1}^k M_i$ is a perfect matching of $G - V(C)$.
    As a consequence, we know that $G$ is cycle-conformal.

    Finally, in regards to checking whether $G$ is cycle-conformal if we already know that $H_1, \ldots , H_k$ are cycle-conformal, we must only confirm algorithmically that the second condition is fulfilled.
    For this purpose, we suggest the following na\"ive method:

    For each edge $e = uv \in F$, subdivide $e$ once with the vertex $v_e$, delete the edge $uv_e$, and check whether there exists an even $u$-$v_e$-path in the resulting graph, which can be done in $\mathcal{O}(m)$-time~\cite{LaPaughP1984Evenpath}.
    This then clearly results in a runtime in $\mathcal{O}(m^2)$.
\end{proof}

\section{Non-planar, Pfaffian braces}\label{sec:braces}
Our first goal is to prove that $C_4$ is the unique Pfaffian, cycle-conformal brace.
We start by giving a simple proof of \zcref{thm:c4uniqueplanar}, since it uses an idea that will play a key role in the proof of the main theorem in this section and illustrates why planar graphs are convenient to work with when trying to determine whether a graph is cycle-conformal.

\begin{lemma}\label{lem:planarbraces}
    There are no planar, cycle-conformal braces with more than four vertices.
\end{lemma}
\begin{proof}
    Let $G$ be a planar, cycle-conformal brace with more than four vertices.
    \zcref{thm:extendabilitybasics} tells us that $G$ is 3-connected.
    Let $u,v \in V(G)$ be two vertices of the same colour class in $G$.
    There exist three $u$-$v$-paths $P_1,P_2,P_3$ in $G$ that are internally disjoint according to \zcref{thm:menger}.
    Hence $H \coloneqq P_1 \cup P_2 \cup P_3$ is a bisubdivision of $K_{2,3}$.
    
    Since $G$ is planar and $H \subseteq G$, the deletion of $H$ from a plane drawing of $G$ divides the plane up into three regions.
    We let $X_1$, $X_2$, and $X_3$ be the three subsets of $V(G) \setminus V(H)$ drawn into these regions.
    Accordingly, we have $\bigcup_{i=1}^3 X_i \cup V(H) = V(G)$.
    As $G$ is a brace, $|V(G)|$ is even.
    Since $H$ contains an odd number of vertices, at least one set among $X_1,X_2,X_3$ must be odd.

    Assume w.l.o.g.\ that $X_1$ is odd.
    According to our definition for $X_1,X_2,X_3$, there exist two distinct $i,j \in [3]$, such that $P_i \cup P_j$ is an even cycle and $P_i \cup P_j$ separates $X_1$ from $X_2 \cup X_3 \cup ( V(H) \setminus V(P_i \cup P_j) )$.
    Since $X_1$ is odd, it cannot be covered by a perfect matching and thus $P_i \cup P_j$ is not a conformal cycle.
\end{proof}

From this lemma it is easy to deduce that the only interesting braces we have to consider are non-planar, Pfaffian braces.
Here our general strategy will be to pluck apart such a brace using \zcref{thm:pfaffian}, find a problematic structure as in the proof of \zcref{lem:planarbraces}, and then show that the trisum operation essentially preserves the properties of the problematic structure.
The following two lemmas due to McCuaig are particularly helpful in this endeavour.

\begin{lemma}[McCuaig~\cite{McCuaig2004Polyas}]\label{lem:facial4cycledeletion}
    If $C$ is a facial cycle of a planar brace $G$, then $G - V(C)$ is matching covered.
\end{lemma}

\begin{lemma}[McCuaig~\cite{McCuaig2004Polyas}]\label{lem:separating4cycle}
    Let $G_1$ and $G_2$ be bipartite such that $G_1 \cap G_2$ is a 4-cycle.
    If $G_1 \cup G_2$ is a brace then $G_1$ and $G_2$ are braces.
\end{lemma}

We will need the following rather specific lemma about perfect matchings containing two edges of a 4-cycle derived from the two results above.

\begin{lemma}\label{lem:matchingwithface}
    Let $G$ be a planar brace containing a 4-cycle $C$, let $e \in E(G)$ be disjoint from $V(C)$, and let $ab, cd \in E(C)$ be disjoint.
    Then there exists a perfect matching $M$ of $G$ with $e, ab, cd \in M$.
\end{lemma}
\begin{proof}
    Independent of whether or not $C$ bounds a face in a plane drawing of $G$, there exist two bipartite graphs $G_1$ and $G_2$ with $G_1 \cap G_2 = C$ such that $G_1 \cup G_2 = G$.
    Since $G$ is a brace, $G_1$ and $G_2$ are therefore also braces according to \zcref{lem:separating4cycle} and furthermore these two graphs have a plane drawing in which $C$ bounds a face.
    We may choose these two graphs such that $e \in E(G_1)$.
    Both $G_1 - V(C)$ and $G_2 - V(C)$ are matching covered according to \zcref{lem:facial4cycledeletion}.
    Thus there exists a perfect matching $M_1$ of $G_1 - V(C)$ with $e \in M_1$ and a perfect matching $M_2$ of $G_2 - V(C)$.
    The perfect matching we seek is therefore $M_1 \cup M_2 \cup \{ ab, cd \}$.
\end{proof}

We follow this up with a result that tells us that the vertices of a $C_4$ in a planar brace are well-connected to each other via paths.

\begin{lemma}\label{lem:disjoint4cyclepaths}
    Let $G$ be a planar brace other than $C_4$ and let $C$ be a 4-cycle in $G$ with $ab, cd \in E(C)$ being disjoint edges.
    Then there exists an $a$-$b$-path $P$ and a $c$-$d$-path $Q$ in $G - E(C)$ such that $P$ and $Q$ are disjoint.
\end{lemma}
\begin{proof}
    We assume without loss of generality that $a,c$ are white and $b,d$ are black vertices in $G$.
    Let $G'$ be $G$ if $C$ is a facial cycle of $G$.
    Otherwise, there exist two planar, bipartite graphs $G_1$ and $G_2$, both not isomorphic to $C_4$, that each have $C$ as a facial cycle such that $G_1 \cup G_2 = G$ and $G_1 \cap G_2 = C$.
    According to \zcref{lem:separating4cycle}, since $G$ is a brace, both $G_1$ and $G_2$ are braces and we can let $G'$ be $G_1$.
    Thus $C$ is a facial cycle in the planar brace $G'$ which is not isomorphic to $C_4$.

    We claim that there exist four pairwise distinct vertices $a',b',c',d' \in V(G')$ such that $aa',bb',cc',dd' \in E(G')$.
    Since $G'$ is a brace and not isomorphic to $C_4$, it is 3-connected.
    In particular, each vertex in $C$ has a neighbour outside of $C$.
    As candidates for $a',b',c',d'$, we choose one of these neighbours arbitrarily for each vertex in $V(C)$.
    Due to the colouring of $V(C)$, if these four vertices are not be distinct, we have $a' = c'$ or $b' = d'$.
    Suppose that $a' = c'$, then $a'$ is a cut-vertex in $G' - V(C)$ due to the planarity in $G'$.
    This however directly contradicts \zcref{lem:facial4cycledeletion}, as $G' - V(C)$ must be 2-connected according to \zcref{thm:extendabilitybasics} and \zcref{lem:facial4cycledeletion}.
    The same argument applies if $b' = d'$.
    Thus our choice of candidates satisfies our claim.
    
    Using Menger's theorem we ask for two disjoint $\{ a', d' \}$-$\{ b', c' \}$-paths in $G' - V(C)$, which is 2-connected as we just noted.
    Due to $C$ being a facial cycle and $G' - V(C)$ being 2-connected, the vertices $a',b',c',d'$ are found together on a facial cycle of $G' - V(C)$, in the stated order.
    Thus, in this cycle we can find the desired $a'$-$b'$- and $c'$-$d'$-path, which we can extend to an $a$-$b$-path and a $c$-$d$-path via $aa'$, $bb'$, $cc'$, and $dd'$, satisfying the statement of our lemma.
\end{proof}

We now show that we can find a very specific trisum in any non-planar Pfaffian brace.

\begin{lemma}\label{lem:planarsummands}
    Let $G$ be a non-planar, Pfaffian brace other than the Heawood graph.
    Then $G$ is the trisum of three Pfaffian braces $G_1,G_2,G_3$, each distinct from $C_4$, at a 4-cycle $C$ such that both $G_1$ and $G_2$ have a plane drawing in which $C$ is a facial cycle.
\end{lemma}
\begin{proof}
    We start by observing that the weaker claim that $G$ is the trisum of three Pfaffian braces $G_1$, $G_2$, and $G_3$, each distinct from $C_4$, at a 4-cycle $C$ such that both $G_1$ and $G_2$ are planar is an easy consequence of the tree-like structure that non-planar, Pfaffian braces exhibit according to definition of the trisum operation and \zcref{thm:pfaffian}.
    If both $G_1$ and $G_2$ already have such a planar drawing, then we are done.
    Thus we may suppose w.l.o.g.\ that $G_1$ cannot be drawn in the plane such that $C$ bounds a face of $G_1$.
    According to \zcref{lem:separating4cycle}, we may split $G_1$ along $C$ into two planar braces $G_1'$ and $G_2'$, each permitting a plane drawing in which $C$ is a facial cycle.

    \begin{claim}\label{claim:unionbrace}
        The graph $G_3' \coloneqq G_2 \cup G_3$ is a Pfaffian brace.
    \end{claim}
    \emph{Proof of \zcref{claim:unionbrace}:}
    First, we note that due to the definition of being Pfaffian relying on the existence of a certain type of orientation, this property is in particular closed under taking subgraphs.
    Thus, since $G_3' \subseteq G \cup C$ and $G \cup C$ is Pfaffian according to \zcref{thm:pfaffian}, we know that $G_3'$ is Pfaffian.
    Hence all that remains to be shown is that $G_3'$ is 2-extendable.
    
    Let $e,f \subseteq G_3'$ be two disjoint edges, let $V(C) = \{ a,b,c,d \}$, with $ab, bc, cd, da \in E(C)$.
    If further $e,f \subseteq E(G_2)$, we let $M_2$ be a perfect matching of $G_2$ with $e,f \in M_2$, which must exist since $G_2$ is 2-extendable.
    In turn, we let $M_3$ be a perfect matching of $G_3$ with $ab, cd \in M_3$.
    Thus $M_2 \cup (M_3 \setminus \{ ab, cd \} )$ is a perfect matching of $G_3'$ containing $e,f$.
    An analogous argument takes care of the case in which $e,f \in E(G_3)$.

    We may therefore assume w.l.o.g.\ that $e \in E(G_2)$, $f \in E(G_3)$, and $\{ e,f \} \cap E(C) = \emptyset$.
    Accordingly, both $e$ and $f$ have at most one endpoint in $V(C)$.

    Suppose that $e$ is disjoint from $V(C)$.
    Since $G_2$ is planar, there exists a perfect matching $M_2'$ of $G_2$ with $e, ab, cd \in M_2'$, according to \zcref{lem:matchingwithface}.
    As $G_3$ is matching covered, there also exists a perfect matching $M_3'$ of $G_3$ containing $f$.
    Thus, $(M_2' \setminus \{ ab, cd \}) \cup M_3'$ is the perfect matching of $G_3'$ we want.
    
    Otherwise, if $e$ does have an endpoint in $V(C)$, we suppose w.l.o.g.\ that $e \cap V(C) = \{ a \}$.
    If $f$ is disjoint from $V(C)$, we let $M_2''$ be a perfect matching of $G_2$ containing $e, cd$ and let $M_3''$ be a perfect matching of $G_3$ containing $f, ab$.
    The desired perfect matching of $G_3'$ is therefore $(M_2'' \setminus \{ cd \}) \cup (M_3'' \setminus \{ ab \})$.

    Finally, should $f \cap V(C) \subseteq \{ b,d \}$, we assume w.l.o.g.\ that $f \cap V(C) = \{ d \}$.
    
    We then let $N_2$ be a perfect matching of $G_2$ containing $e, cd$ and $N_3$ be a perfect matching of $G_3$ containing $f, ab$.
    This yields $(N_2 \setminus \{ cd \}) \cup (N_3 \setminus \{ ab \})$ as the matching we seek.
    Thus $G_3'$ is 2-extendable.
    \hfill$\blacksquare$
    
    We conclude that $G_1'$, $G_2'$, and $G_2 \cup G_3$ have the desired properties.
\end{proof}

Furthermore, we will want to know that any edge contained in a 4-cycle is contained in all three summands of a trisum $G$ if no edge of the 4-cycle used to perform this trisum is missing from $G$.

\begin{lemma}\label{lem:trisum4cycleedges}
    Let $G$ be a bipartite graph that is the trisum of $G_1,G_2,G_3$ at a 4-cycle $C$ with $C \subseteq G$ and let $e \in E(G)$ be an edge contained in a 4-cycle $C_e$ of $G$.
    Then there exists an $i \in [3]$ such that $e$ is contained in a 4-cycle $C_e'$ in $G_i$ and we have $V(C_e) \cap V(C_e') = V(C_e) \cap V(G_i)$.
\end{lemma}
\begin{proof}
    Let $C_e \subseteq G$ be a 4-cycle with $e \in E(C_e)$, then if $C_e \subseteq G_i$ for some $i \in [3]$, we are clearly done.
    Thus, we must instead have $V(C_e) \cap V(C) \neq \emptyset$.
    In particular, letting $V(C) = \{ a,b,c,d \}$ with $a,c$ occupying the same colour class in $G$, the set $V(C_e) \cap V(C)$ is either $\{ a,c \}$ or $\{ b,d \}$, as any other option immediately implies that there exists an $i \in [3]$ with $C_e \subseteq G_i$.
    Furthermore, there must exist two distinct $i \in [3]$ such that $|(V(G_i) \cap V(C_e)) \setminus V(C)| = 1$ and $|(E(G_i) \cap E(C_e)) \setminus E(C)| = 2$.
    However, in this case, let $i \in [3]$ be the one of these two options such that $e \in E(G_i)$.
    Let $u \in V(G_i) \setminus V(C)$ be the vertex in $V(C_e)$ that has two neighbours in $C$.
    We assume w.l.o.g.\ that these neighbours are $\{ a,c \}$.
    Then clearly $u,a,b,c$ induce a 4-cycle in $G_i$ that contains $e$ and we are done.
\end{proof}

With all of our tools gathered, we can now prove \zcref{thm:c4uniquepfaffian}, which extends \zcref{lem:planarbraces}.

\begin{theorem}\label{thm:main1}
    No Pfaffian brace with more than four vertices is cycle-conformal.
\end{theorem}
\begin{proof}
    Suppose there exists a Pfaffian brace $G$ that has more than four vertices and is cycle-conformal.
    According to \zcref{lem:planarbraces}, the graph $G$ cannot be planar.
    The fact that the Heawood graph is not cycle-conformal is demonstrated in \zcref{fig:heawood}.

    Thus $G$ is a non-planar Pfaffian brace that is distinct from the Heawood graph.
    According to \zcref{lem:planarsummands}, $G$ is therefore the trisum of three Pfaffian braces $G_1,G_2,G_3$, each distinct from $C_4$, at a 4-cycle $C$ such that both $G_1$ and $G_2$ have a plane drawing in which $C$ is a facial cycle.
    
    Let $V(C) = \{ a,b,c,d \}$ such that $a,c$ occupy the same colour class of $G_1$, let $a'$ be a neighbour of $a$, and let $c'$ be a neighbour of $c$ in $G_1$.
    As argued earlier, due to the 3-connectivity and planarity of $G_1$, we know that $a' \neq c'$.
    Consider the facial cycle $C'$ in $G_1 - V(C)$ whose corresponding face contains the face bounded by $C$ in $G_1$.
    We must have $a',c' \in V(C')$ and we let the two distinct $a'$-$c'$-paths in $C'$ be $P_2$ and $P_3$ respectively.
    Note that $a',c'$ occupy the same colour class in $G_1$ and thus $P_2$ and $P_3$ both have even length.
    We let $P_1$ be the path in $G_1$ containing the vertices $a',a,b,c,c'$, which has even length.
    Thus $H \coloneq P_1 \cup P_2 \cup P_3$ is a bisubdivision of $K_{2,3}$.

    Since $G_1$ is planar and $H \subseteq G_1$, the deletion of $H$ from a planar drawing of $G_1$ divides the plane up into three regions.
    We let $X_1$, $X_2$, and $X_3$ be the three subsets of $V(G_1) \setminus V(H)$ that are drawn into these regions.
    By construction of $H$, we may name these such that we have $X_1 = \{ d \}$, $X_2 = \emptyset$, and $X_3 = V(G_1) \setminus (V(H) \cup \{ d \})$.
    We note that therefore $X_1$ is the unique set of odd size, as $|V(H)|$ is odd and we have $\bigcup_{i=1}^3 X_i \cup V(H) = V(G_1)$.

    A core property for the proof of \zcref{lem:planarbraces} is that the three sets $X_1, X_2, X_3$ are separated from each other in $G_1 - V(H)$.
    Note that if we let $X_1^\star = X_1 \cup (G \setminus V(G_1))$, this property also holds for $X_1^\star,X_2,X_3$.
    Thus, if we have $ab,bc \in E(G)$, we are now done via the same argument as in the proof of \zcref{lem:planarbraces}, since in this case we have $H \subseteq G$.

    We solve the more general case by finding paths in $G_2$ and $G_3$ that stand in for these edges.
    Using \zcref{lem:disjoint4cyclepaths}, we let $P_1^2$ be an $a$-$b$-path in $G_2$ that avoids $c,d$ as internal vertices and does not use any edge in $E(C)$.
    
    Next, we repeatedly apply \zcref{lem:trisum4cycleedges} to $G_3$ to find a collection of planar graphs $G_1', \ldots , G_\ell'$ which allow us to construct $G_3$ via repeated trisums and such that there exists an $i \in [\ell]$ and a 4-cycle $C' \subseteq G_i'$ with $bc \in E(C')$ such that $V(C) \cap V(C') = V(C) \cap V(G_i')$.
    We can then apply \zcref{lem:disjoint4cyclepaths} to find a $b$-$c$-path $P_1^3$ in $G_i'$ that avoids both $V(C') \setminus \{ b,c \}$ and the edge $bc$.
    Both $P_1^2$ and $P_1^3$ have odd length and intersect only at $b$.
    Therefore, $P_1^2 \cup P_1^3$ is an $a$-$c$-path of even length.

    This allows us to define $H' \coloneqq (H - b) \cup P_1^2 \cup P_1^3$, where we note that $|V(H')|$ and $|V(H)|$ have the same parity, we have $b \in V(H')$, and the unique $a'$-$c'$-path $P_1'$ in $H'$ using $b$ has even length, as did $P_1$ in $H$.
    We may thus let $X_1' = X_1^\star \setminus V(P_1^2 \cup P_1^3)$.
    Since $a,b,c \in V(H')$, the sets $X_1',X_2, X_3$ are pairwise separated in $G - V(H')$.
    Via analogous arguments to those presented in the proof of \zcref{lem:planarbraces}, we conclude that $G$ is not cycle-conformal, as $|X_1'|$ is odd.
\end{proof}

\section{Cycle-conformality and tight cuts}\label{sec:tightcuts}
Now that we have established that there is exactly one cycle-conformal, Pfaffian brace, we consider the possibility of extending results about cycle-conformality over tight cuts.
Our two goals in this section are to prove that one can indeed propagate this property through tight cuts in cubic, bipartite graphs and secondly, we want to extend our characterisation of cycle-conformal, Pfaffian braces to cycle-conformal, Pfaffian, matching covered graphs.

We start by considering two more helpful results due to McCuaig.
A matching $F$ in a graph $G$ is called \emph{induced} if no two endpoints of distinct edges of $F$ are adjacent in $G$.

\begin{lemma}[McCuaig~\cite{McCuaig2000Even}]\label{lem:tightcutshape}
    Let $G$ be a 3-connected, cubic, bipartite graph.
    A non-trivial cut in $G$ is tight if and only if it is an induced matching consisting of 3 edges.
\end{lemma}

\begin{lemma}[McCuaig~\cite{McCuaig2000Even}]\label{lem:tightcutpreserve}
    Let $G$ be a 3-connected, cubic, bipartite graph with a non-trivial tight cut $\partial(X)$.
    Then both tight cut contractions associated with $\partial(X)$ are cubic, 3-connected, and bipartite.
\end{lemma}

This pair of results now directly allows us to reach our first goal of this section.

\begin{lemma}\label{lem:tightcutspreservecycconf}
    Let $G$ be a cubic, matching covered, bipartite graph with a tight cut $\partial(X)$ and let $G_1,G_2$ be the two distinct tight cut contractions associated with $\partial(X)$.
    Then, if $G_1$ and $G_2$ are cycle-conformal, then $G$ is cycle-conformal.
\end{lemma}
\begin{proof}
    Let $C$ be an even cycle in $G$.
    Suppose first that there exists an $i \in [2]$ such that $C \subseteq G_i$, then let $M_i$ be the perfect matching in $G_i - V(C)$ that must exist.
    Let $e = uc \in M_i$ be the unique edge incident to the contraction vertex $c$ in $G_i$ and let $uv \in \partial(X)$ be an edge incident to $u$ in $\partial(X)$, which must exist since $u$ is incident to $c$ in $G_i$.
    Note that this implies that $v \in V(G_{3-i})$ and $vc \in E(G_{3-i})$, where $c$ is the contraction vertex in $G_{3-i}$.
    Since $G_{3-i}$ is matching covered, there exists a perfect matching $M_{3-i}$ of $G_{3-i}$ that uses $vc$.
    Now, $(M_1 \cup M_2 \cup \{ uv \}) \setminus \{ uc, vc \}$ is a perfect matching of $G - V(C)$.

    We may therefore suppose that $E(C) \cap \partial(X) \neq \emptyset$.
    In particular, due to \zcref{lem:tightcutshape}, we have $|E(C) \cap \partial(X)| = 2$, as any cycle contains an even number of edges of any edge-cut in a graph.
    Let $u_1u_2 \in \partial(X) \setminus E(C)$ be the one edge in the cut that is not found in $E(C)$, with $u_i \in V(G_i)$ for both $i \in [2]$.
    Further, note that due to $\partial(X)$ being an induced matching according to \zcref{lem:tightcutshape}, for both $i \in [2]$, the subgraph $C_i \subseteq G_i$ corresponding to $C$ in $G_i$.
    We note that $u_ic \in E(G_i)$, $C_i$ is an even cycle, and $u_i \not\in V(C_i)$.

    Since $G_i$ is cycle-conformal for both $i \in [2]$, the graph $G_i - V(C_i)$ contains a perfect matching $M_i$ that must use $u_ic$, since $c$ has degree 1 in $G_i - V(C_i)$.
    This immediately implies that $(M_1 \cup M_2 \cup \{ u_1u_2 \}) \setminus \{ u_1c, u_2c \}$ is a perfect matching of $G - V(C)$ and thus we are done.
\end{proof}

\zcref{lem:tightcutpreserve}, \zcref{lem:tightcutspreservecycconf}, and \zcref{thm:tightcutpreserveeasy} together imply \zcref{thm:cubictightcuts}.
It is reasonable to ask whether \zcref{lem:tightcutspreservecycconf} can be extended to $\ell$-regular, matching covered, bipartite graphs for $\ell \geq 4$.
However, this can easily be seen to be false through sticking together two copies of $K_{\ell,\ell}$.

\begin{lemma}\label{lem:counterexample}
    Let $\ell \geq 4$ be an integer and let $G$ be the graph derived from two copies $H_1,H_2$ of $K_{\ell,\ell}$ by deleting a white vertex of $H_1$, deleting a black vertex of $H_2$, and subsequently joining all black vertices of $H_1$ to all white vertices of $H_2$ via edges (see \zcref{fig:counterexample} for an example with $\ell = 4$).
    Then $G$ is matching covered, bipartite, and not cycle-conformal.
    However $G$ contains a tight cut such that both associated tight cut contractions are cycle-conformal braces.
\end{lemma}
\begin{proof}
For both $i \in [2]$ let the white vertices of $H_i$ be $w_1^i, \ldots , w_\ell^i$ and the black vertices of $H_i$ be $b_1^i, \ldots , b_\ell^i$.
We assume that $V(H_1 \cup H_2) \setminus V(G) = \{  w_\ell^2, b_\ell^1 \}$ and let $F \coloneqq E(G) \setminus E(H_1 \cup H_2)$ be the set of edges introduced in the construction of $G$.
It is easy to observe that $G$ is matching covered and bipartite.
Clearly, $F$ is a tight cut and both tight cut contractions associated with $F$ yield $K_{\ell,\ell}$.

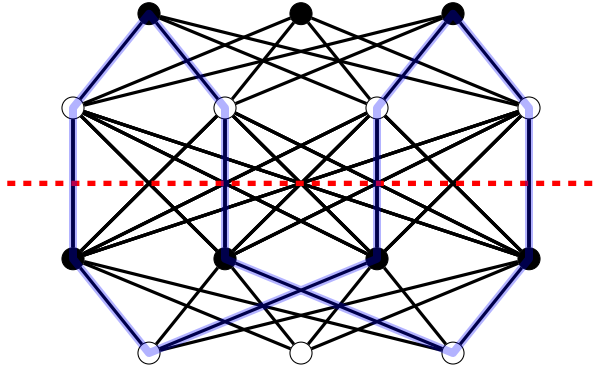
\begin{figure}[ht]
        \centering
            \begin{tikzpicture}

            \node (A1) at (1,5.25) [draw, circle, scale=0.8, fill] {};
            \node (A2) at (3,5.25) [draw, circle, scale=0.8, fill] {};
            \node (A3) at (5,5.25) [draw, circle, scale=0.8, fill] {};
            
            \node (B1) at (0,4) [draw, circle, scale=0.8] {};
            \node (B2) at (2,4) [draw, circle, scale=0.8] {};
            \node (B3) at (4,4) [draw, circle, scale=0.8] {};
            \node (B4) at (6,4) [draw, circle, scale=0.8] {};
            
            \node (C1) at (0,2) [draw, circle, scale=0.8, fill] {};
            \node (C2) at (2,2) [draw, circle, scale=0.8, fill] {};
            \node (C3) at (4,2) [draw, circle, scale=0.8, fill] {};
            \node (C4) at (6,2) [draw, circle, scale=0.8, fill] {};

            \node (D1) at (1,0.75) [draw, circle, scale=0.8] {};
            \node (D2) at (3,0.75) [draw, circle, scale=0.8] {};
            \node (D3) at (5,0.75) [draw, circle, scale=0.8] {};

            \node (Z1) at (-1,3) [draw=none] {};
            \node (Z2) at (7,3) [draw=none] {};

            \foreach\i in {1,2,3}
			{
				\foreach\j in {1,2,3,4}
                {
				    \path (A\i) edge[very thick] (B\j);
                    \path (D\i) edge[very thick] (C\j);
                    \foreach\k in {1,2,3,4}
                    {
                        \path (B\j) edge[very thick] (C\k);
                    }
			    }
            }

            \path (Z1) edge[dashed,red, line width=2pt] (Z2);

            \draw[blue, line width=3pt, opacity=0.3, line join=round]
                (1,5.25) -- (0,4) -- (0,2) -- (1,0.75) -- (4,2) -- (4,4) -- (5,5.25) -- (6,4) -- (6,2) -- (5,0.75) -- (2,2) -- (2,4) -- cycle;
            ;

		\end{tikzpicture}
     \caption{One of the graphs described in \zcref{lem:counterexample} with $\ell = 4$. The tight cut $F$ is indicated by the red dashed line and the problematic cycle $C$ is traced in blue.}
     \label{fig:counterexample}
\end{figure}

We now consider the cycle $C \subseteq G$ formed by visiting the vertices
\[ b_1^1, w_1^1, b_1^2, w_1^2, b_{\ell-1}^2, w_{\ell-1}^1, b_{\ell-1}^1, w_\ell^1, b_\ell^2, w_2^2, b_2^1, w_2^1, b_1^1 \]
in order (see \zcref{fig:counterexample} for reference).
Since $G$ is bipartite, $C$ is even and thus $G - V(C)$ must have a perfect matching.
However, $V(G - V(C))$ contains $\ell - 3$ black vertices $b_2^1, \ldots , b_{\ell-2}^1$, which only have $\ell - 4$ neighbours $w_3^1, \ldots , w_{\ell-3}^1$ in $G - V(C)$, contradicting the existence of a perfect matching in this graph.
Thus $G$ has all of the desired properties.
\end{proof}

We now turn to the issue of characterising matching covered, Pfaffian, bipartite, cycle-conformal graphs and thus the proof of \zcref{thm:pfaffcycconfcharacterisation}.
First we show that any such graph is planar.

\begin{theorem}
    Every matching covered, Pfaffian, bipartite, cycle-conformal graph is planar.
\end{theorem}
\begin{proof}
    Suppose that there exists a counterexample to this and thus a non-planar, matching covered, Pfaffian, bipartite, cycle-conformal graph.
    Thanks to \zcref{thm:tightcutpreserveeasy}, we know that all braces belonging to this graph are isomorphic to $C_4$.
    According to \zcref{lem:findbrickbrace} we can find a tight cut in this graph that only splits off a copy of $C_4$ and furthermore, again using \zcref{thm:tightcutpreserveeasy}, we note that the other tight cut contraction that results remains cycle-conformal.
    By iterating this procedure, we may therefore assume that we can reduce our counterexample to a graph $G$ that is non-planar, matching covered, Pfaffian, bipartite, cycle-conformal, and contains a tight cut $\partial(X)$ such that one of the two tight cut contractions $G_1,G_2$ corresponding to $\partial(X)$, say $G_2$, is isomorphic to $C_4$ and the other is planar, matching covered, bipartite, and cycle-conformal.
    In particular, the vertex in $G_2$ that is incident to none of the edges of $\partial(X)$ in $G$ must have degree 2 in $G$ as well.
    
    We choose $X$ such that $X \subseteq G_2$ and thus $|X| = 3$.
    Specifically, we let $X = \{ u,v,w \}$, with $v$ being the vertex of degree 2 in both $G_2$ and $G$ that is not incident to any of the edges in $\partial(X)$.
    We call the contraction vertex $c$ in both $G_1$ and $G_2$.
    
    Since $G$ is non-planar it contains a subdivision $H$ of $K_5$ or $K_{3,3}$ as a subgraph.
    Let $V' \subseteq V(H)$ be the vertices of degree at least 3 in $H$.
    Clearly, if $V' \cap X = \emptyset$, then a subdivision of the graph corresponding to $H$ also exists in $G_1$, contradicting the fact that $G_1$ is planar.
    Thus, we must have $V' \cap X \neq \emptyset$.
    Since $v$ has degree two, we have $v \not\in V'$ and therefore know that $1 \leq |V' \cap X| \leq 2$.

    Suppose first that $|V' \cap X| = 1$.
    This vertex is either $u$ or $w$.
    However, by definition of tight cut contractions, the neighbourhood of both $u$ and $w$ is a subset of the neighbourhood of $c$ in $G_1$.
    Thus, we can again find a subdivision of the graph corresponding to $H$ in $G_1$, once more contradicting its planarity.

    We therefore now know that $V' \cap X = \{ u,w \}$.
    Now, observe that $\{ u,w \}$ is a separator in $G$, which in particular separates $v$ from the remainder of the graph.
    Furthermore, since $u,w$ are found in $H$ then, independent of whether $H$ is a subdivision of $K_5$ or $K_{3,3}$, $H$ contains a cycle $C$ containing both $u$ and $w$.
    In particular, $C$ can be chosen such that it avoids $v$, if $v \in V(H)$, since $v$ would lie in a subdivided edge of $K_5$ or $K_{3,3}$, any of which can be avoided in a cycle using $u$ and $w$ in $H$.
    We conclude that $G - V(C)$ cannot contain a perfect matching, since $v$ is an isolated vertex in this graph, and thus $G$ is not cycle-conformal, a contradiction.
\end{proof}

Now, mainly for the convenience of the reader, we state one of the many characterisations of matching covered, planar, bipartite, cycle-conformal graphs (see ~\cite{GuoZ2004Reducible,KlavzarS2012Characterization,ZhangL2012Computing,Pause2022Planar,Kuske2024Characterization}).
Here we use one of the characterisations found by Kuske~\cite{Kuske2024Characterization}.

Let $H$ be a graph and let $e = uv \in E(H)$ be an edge of $H$.
A \emph{bisubdivision (of $H$ at $e$)} is a subdivision of $H$ that is achieved by subdividing the edge $e$ with an even number of subdivision vertices.
We say that $H'$ is constructed from $H$ via the \emph{3-path operation (at $e$)} if $H'$ is derived from $H$ by adding the vertices $x_1,x_2 \not\in V(H)$ and the edges $ux_1, x_1x_2, x_2v$.

\begin{theorem}[Kuske~\cite{Kuske2024Characterization}]\label{thm:planarcharacterisation}
    A matching covered, planar, bipartite graph is cycle-conformal if and only if it can be constructed from $C_4$ via bisubdivisions and the 3-path operation.
\end{theorem}

Using \zcref{thm:pfaffcycconfcharacterisation}, we can thus derive the following corollary from \zcref{thm:planarcharacterisation}.

\begin{corollary}\label{cor:Pfaffiancharacterisation}
    A matching covered, Pfaffian, bipartite graph is cycle-conformal if and only if it can be constructed from $C_4$ via bisubdivisions and the 3-path operation.
\end{corollary}

\section{Cycle-conformal, cubic braces}\label{sec:cubic}
We now finally turn to the proof of \zcref{thm:k33issolo}, which resolves the case $\ell=3$ of \zcref{con:main}.
This argument is split into two parts, starting with the proof that cubic, cycle-conformal braces must contain a special subdivision of $K_{3,3}$, which we can then use as the underlying structure for our proof of \zcref{thm:k33issolo}.

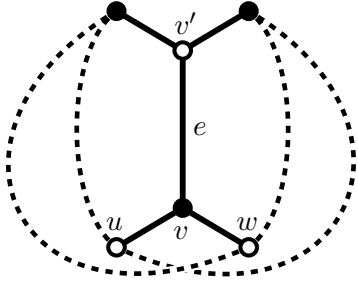
\begin{figure}[ht]
    \centering
        \begin{tikzpicture}[scale=1.25]

            \pgfdeclarelayer{background}
		      \pgfdeclarelayer{foreground}
			
		      \pgfsetlayers{background,main,foreground}

            \begin{pgfonlayer}{background}
            \pgftext{\includegraphics[width=4cm]{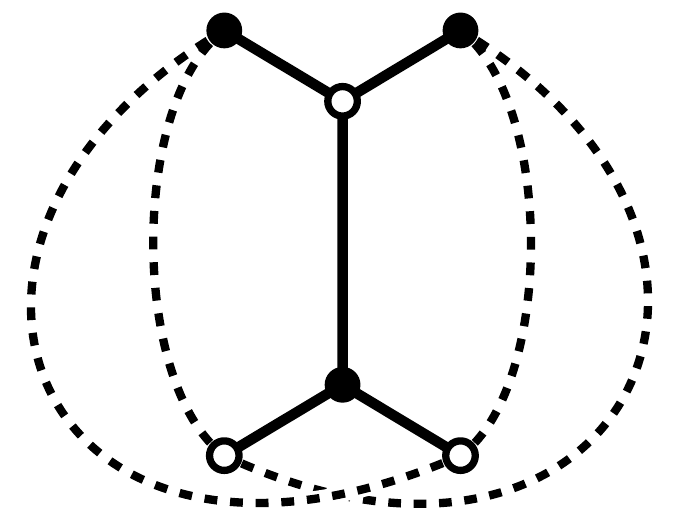}} at (C.center);
            \end{pgfonlayer}{background}
			
            \begin{pgfonlayer}{foreground}

            \node (V) at (0,-1) [draw=none] {$v$};
            \node (VPRIME) at (0.05,1.225) [draw=none] {$v'$};
            \node (e) at (0.2,0.1) [draw=none] {$e$};
            \node (U) at (-0.7,-0.9) [draw=none] {$u$};
            \node (W) at (0.7,-0.9) [draw=none] {$w$};

            \end{pgfonlayer}{foreground}
        
        \end{tikzpicture}
    \caption{An illustration of the kind of $K_{3,3}$ we purport to exist in \zcref{lem:k33subdivision}. The five bold lines represent single edges and the dashed lines represent non-trivial paths.}
    \label{fig:K33}
\end{figure}

\begin{lemma}\label{lem:k33subdivision}
Let $G$ be a cubic, cycle-conformal brace.
Then $G$ contains a subdivision $H$ of $K_{3,3}$ in which there exists an edge $e$ whose endpoints $v,v'$ have degree 3 in $H$ and such that all edges incident to $v$ or $v'$ in $H$ have the property that both of their endpoints have degree 3 in $H$ (see \zcref{fig:K33} for an illustration of such a subdivision).
\end{lemma}
\begin{proof}
    Let $A,B$ be the two colour classes of $G$, where we let $A$ be the white vertices and $B$ be the black vertices.
    Further, let $e = vv'$ be an arbitrary edge in $G$ with $v \in A$ and $v' \in B$, and let $u,w$ be the two neighbours of $v$ distinct from $v'$.
    
    As $G$ is 3-connected, there exist three internally disjoint $v$-$v'$-paths in $G$.
    Due to $G$ being cubic, one of these paths corresponds to $e$, another path $P_u$ contains $u$, and the third path $P_w$ contains $w$.
    Further, we know that the two neighbours of $v'$ distinct from $v$ are contained in $V(P_u \cup P_w)$.
    Let $u' \in N(v') \cap V(P_u)$ and $w' \in N(v') \cap V(P_w)$ be these two neighbours.

    Once more exploiting the 3-connectivity of $G$, we note that there also exist three internally disjoint $u$-$w$-paths in $G$.
    One of these must consist exactly of $u$, $v$, and $w$.
    For the remaining two paths $Q_1,Q_2$, we let $Q_1$ be the one that shares an edge with $P_u$ (see \zcref{fig:AB} for reference).

    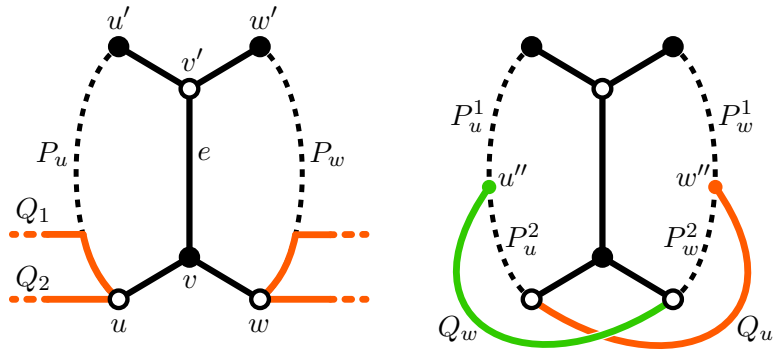
\begin{figure}[ht]
    \centering
        \begin{tikzpicture}[scale=1.25]

            \pgfdeclarelayer{background}
		      \pgfdeclarelayer{foreground}
			
		      \pgfsetlayers{background,main,foreground}

            \begin{pgfonlayer}{background}
            \pgftext{\includegraphics[width=8cm]{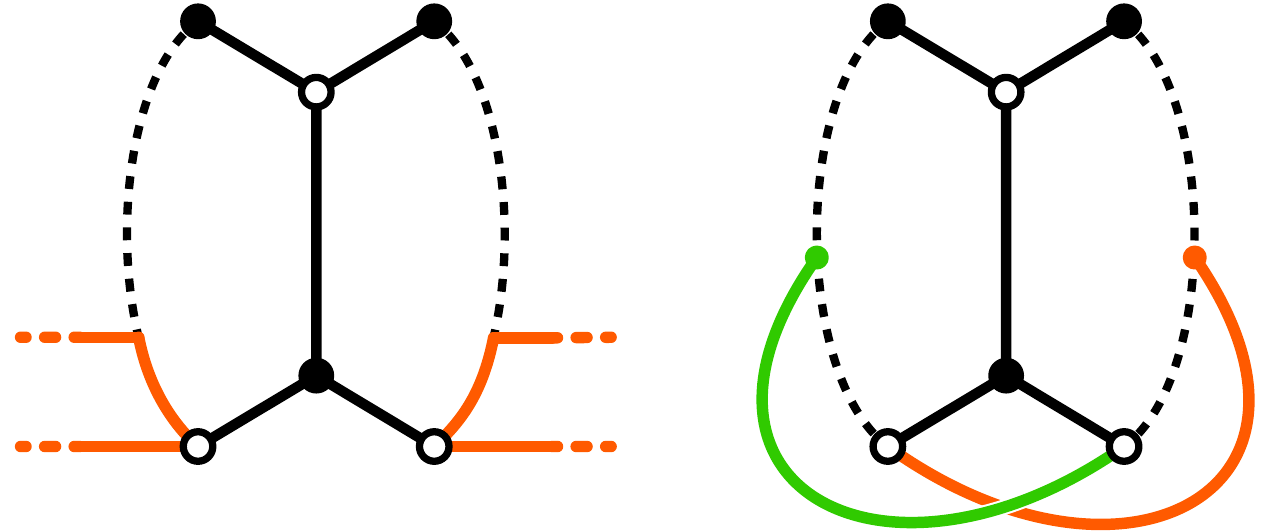}} at (C.center);
            \end{pgfonlayer}{background}
			
            \begin{pgfonlayer}{foreground}

            \node (U) at (-2.75,-1.425) [draw=none] {$u$};
            \node (V) at (-2,-0.95) [draw=none] {$v$};
            \node (W) at (-1.275,-1.425) [draw=none] {$w$};
            \node (UPRIME) at (-2.75,1.85) [draw=none] {$u'$};
            \node (VPRIME) at (-1.975,1.4) [draw=none] {$v'$};
            \node (WPRIME) at (-1.225,1.85) [draw=none] {$w'$};
            \node (Q1) at (-3.65,-0.2) [draw=none] {$Q_1$};
            \node (Q2) at (-3.65,-0.925) [draw=none] {$Q_2$};
            \node (PU) at (-3.45,0.4) [draw=none] {$P_u$};
            \node (PW) at (-0.55,0.4) [draw=none] {$P_w$};
            \node (e) at (-1.85,0.4) [draw=none] {$e$};

            \node (UPRIMEPRIME) at (1.4,0.175) [draw=none] {$u''$};
            \node (WPRIMEPRIME) at (3.3,0.175) [draw=none] {$w''$};
            \node (P1U) at (0.925,0.8) [draw=none] {$P^1_u$};
            \node (P2U) at (1.5,-0.5) [draw=none] {$P^2_u$};
            \node (P1W) at (3.75,0.8) [draw=none] {$P^1_w$};
            \node (P2W) at (3.175,-0.5) [draw=none] {$P^2_w$};
            \node (QW) at (0.825,-1.45) [draw=none] {$Q_w$};
            \node (QU) at (3.95,-1.45) [draw=none] {$Q_u$};
            
            \end{pgfonlayer}{foreground}
        
        \end{tikzpicture}
    \caption{The figure to the left represents the situation in the proof of \zcref{lem:k33subdivision} before \zcref{claim:nobouncing} and the figure to the right represents the graph $H$ we can construct prior to \zcref{claim:noskipping} with the help of \zcref{claim:nobouncing}. In the left figure $Q_1$ and $Q_2$ are indicated by the orange paths. To the right, the green path is $Q_w$ and the orange path is $Q_u$.}
    \label{fig:AB}
    \end{figure}

    We choose $P_u$, $P_w$, $Q_1$, and $Q_2$ such that $|E(P_u \cup P_w \cup Q_1 \cup Q_2)|$ is minimal.
    This allows us to make some further observations on the structure of these paths.

    \begin{claim}\label{claim:nobouncing}
        For both $i \in [2]$ and distinct $x,y \in \{ u,w \}$ there does not exist a non-trivial $V(P_x)$-path in $Q_i$ that is disjoint from $V(P_y)$.
    \end{claim}
    \emph{Proof of \zcref{claim:nobouncing}:}
    Suppose that there does exist a non-trivial $V(P_x)$-path $P$ in $Q_i$ that is disjoint from $V(P_y)$.
    Let $a,b$ be the two endpoints of $P$ and let $P_x'$ be derived from $P_x$ by replacing $aP_xb$ by $P$.
    As $aP_xb$ contains at least one edge not found in $E(P_x' \cup P_y \cup Q_1 \cup Q_2)$, this contradicts the minimality of $P_u,P_w,Q_1,Q_2$.
    \hfill$\blacksquare$

    From this claim, we derive the existence of a $u$-$V(P_w)$-path $Q_u$ that is disjoint from $\{ v,v' \} \cup (V(P_u) \setminus \{ u \})$, which is found in $Q_2$.
    Analogously there also exists a $w$-$V(P_u)$-path $Q_w$ that is disjoint from $\{ v,v' \} \cup (V(P_w) \setminus \{ w \})$.
    In particular, $Q_u$ and $Q_w$ are disjoint.
    Let $w''$ be the endpoint of $Q_u$ that is not $u$ and let $u''$ be the endpoint of $Q_w$ that is not $w$.
    We further note that $w'' \neq w$, as otherwise the cycle $Q_u \cup P_u \cup G[\{ u'v',w' \}] \cup P_w$ is a cycle whose removal isolates $v$ and thus contradicts $G$ being cycle-conformal.
    Analogously, we have $u'' \neq u$.

    From this point onwards we will no longer make use of the minimality of $P_u,P_w,Q_1,Q_2$, as we only needed it to verify the existence of $Q_u$ and $Q_w$.
    
    Let $H$ be the union of $P_u$, $P_w$, $Q_u$, $Q_w$, and the graph $(\{ u,v,w,u',v',w' \} , \{ u'v',v'w',vv',uv,vw \}) \subseteq G$.
    We further divide $P_u$ and $P_w$.
    For both $x \in \{ u,w \}$, let $P_x^1 \coloneqq x'P_xx''$ and $P_x^2 \coloneqq x''P_xx$.
    See \zcref{fig:AB} for a depiction of $H$ with all new objects labelled.
    
    To proceed, we first observe a few facts about reaching upwards to the paths $P_u^1$ and $P_w^1$ in our existing structure.
    Let $G' \coloneqq G - \{ v,v' \}$.

    \begin{claim}\label{claim:noskipping}
        For distinct $x,y \in \{ u,w \}$ the following do not exist:
        \begin{enumerate}
            \item A $V(P_x^2 - x'')$-$V(P_y^1 - y'')$-path in $G'$ that is internally disjoint from $H$.

            \item A $V(Q_y - y'')$-$V(P_x^1 - x'')$-path in $G'$ that is internally disjoint from $H$.
        \end{enumerate}
    \end{claim}
    \emph{Proof of \zcref{claim:noskipping}:}
    This claim is easy to confirm by hand by finding non-conformal cycles (see \zcref{fig:CD} for an illustration).
    \hfill$\blacksquare$

    \begin{figure}[ht]
    \centering
        \begin{tikzpicture}[scale=1.25]

            \pgfdeclarelayer{background}
		      \pgfdeclarelayer{foreground}
			
		      \pgfsetlayers{background,main,foreground}

            \begin{pgfonlayer}{background}
            \pgftext{\includegraphics[width=8cm]{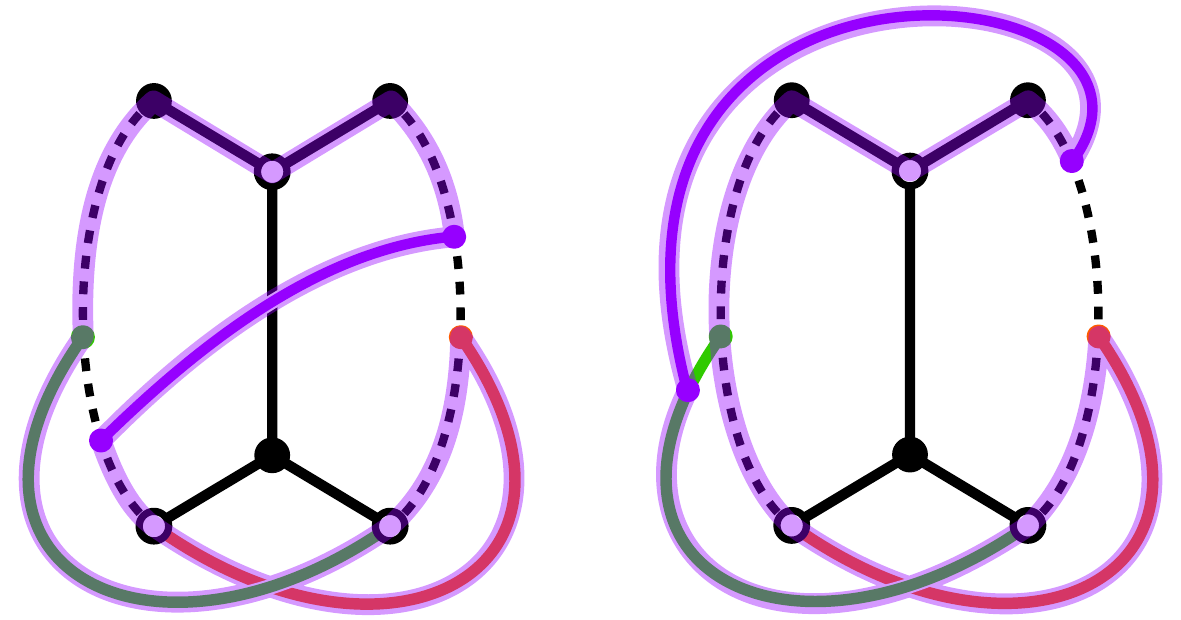}} at (C.center);
            \end{pgfonlayer}{background}
			
            \begin{pgfonlayer}{foreground}

            \node (V) at (-2.175,-1.225) [draw=none] {$v$};

            \node (VAGAIN) at (2.15,-1.225) [draw=none] {$v$};
            
            \end{pgfonlayer}{foreground}
        
        \end{tikzpicture}
    \caption{Two figures illustrating \zcref{claim:noskipping}. The figure to the left contains a $V(P_x^2 - x'')$-$V(P_y^1 - y'')$-path marked in purple and the right figure contains a $V(Q_y - y'')$-$V(P_x^1 - x'')$-path also marked in purple. In both illustrations the vertex $v$ gets isolated by an even cycle. Thus neither represented graph can be cycle-conformal.}
    \label{fig:CD}
    \end{figure}

    In other words, if we want to reach $V(P_x^1 - x'')$ in $G'$ with a path that is internally disjoint from $H$, we must start in $V(P_x^2 \cup Q_x \cup P_y^1)$.
    Furthermore, such a path cannot start in $y''$, as $G$ is cubic, two of the edges incident to $y''$ lie in $P_y$ and the other lies in $Q_x$.
    Having made this observation, our new target for minimisation is the combined length of $P_u^1$ and $P_w^1$.
    
    According to \zcref{claim:noskipping}, the set $V(P_u^1 - u'')$ may be reached from our existing structure in $G'$ via a $V(P_u^2 - u'')$-$V(P_u^1 - u'')$-path $R_1$ that is internally disjoint from $H$ or via a $V(Q_w - u'')$-$V(P_u^1 - u'')$-path $R_2$ that is internally disjoint from $H$.

    \begin{figure}[ht]
    \centering
        \begin{tikzpicture}[scale=1.25]

            \pgfdeclarelayer{background}
		      \pgfdeclarelayer{foreground}
			
		      \pgfsetlayers{background,main,foreground}

            \begin{pgfonlayer}{background}
            \pgftext{\includegraphics[width=8cm]{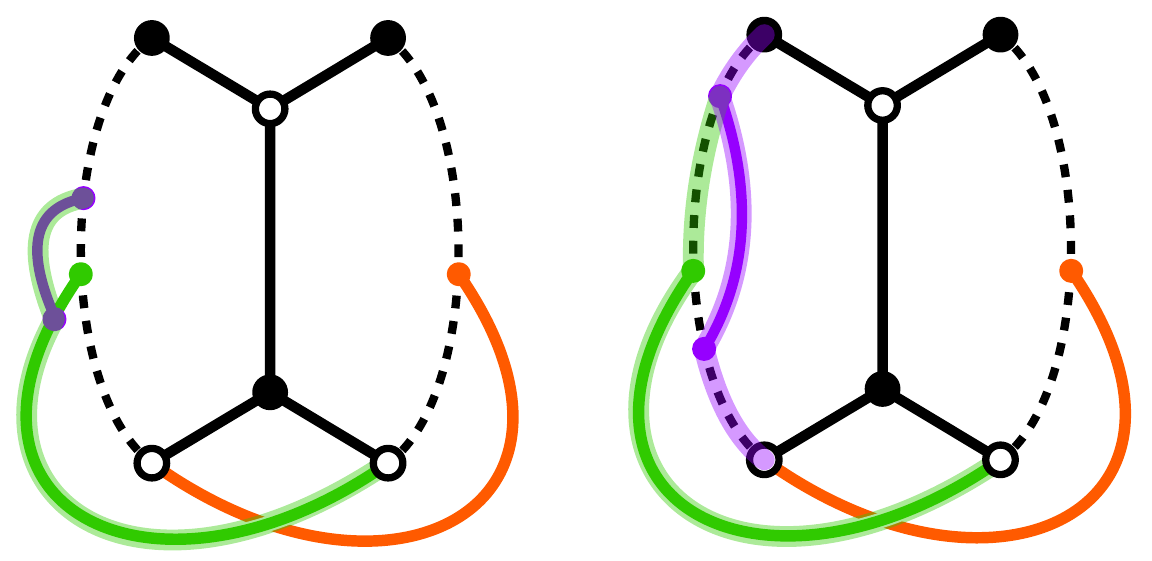}} at (C.center);
            \end{pgfonlayer}{background}
			
            \begin{pgfonlayer}{foreground}

            \node (QW) at (-3.95,-0.2) [draw=none] {$Q_w$};
            \node (R2) at (-4.05,0.3) [draw=none] {$R_2$};

            \node (QWAGAIN) at (0.4825,0.125) [draw=none] {$Q_w$};
            \node (R1) at (1.4,0.3) [draw=none] {$R_1$};
            \node (PU) at (1.2,-0.5) [draw=none] {$P_u$};
            
            \end{pgfonlayer}{foreground}
        
        \end{tikzpicture}
    \caption{Two figures illustrating the minimisation process associated with $R_1$ and $R_2$. In both examples a new candidate for $Q_w$ is highlighted in green and on the right a new candidate for $P_u$ is highlighted in purple.}
    \label{fig:FG}
    \end{figure}
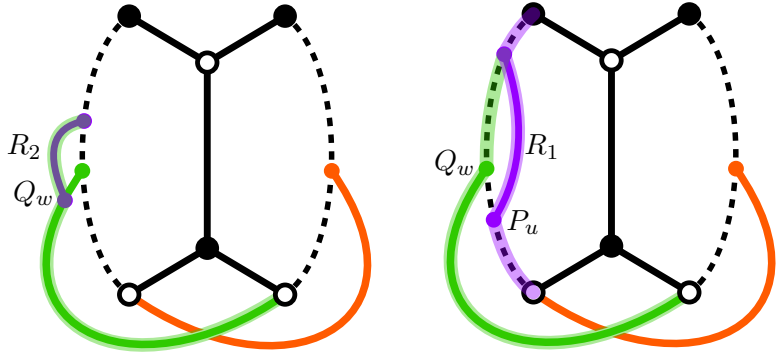

    Suppose first that $R_1$ exists and let $a,b$ be its two endpoints with $b \in V(P_u^1) \setminus \{ u'' \}$.
    We may then let $P_u'$ be constructed from $P_u$ by replacing $aP_ub$ by $R_1$ and let $Q_w'$ be constructed from $Q_w$ by adding $u''P_ub$ to it (see \zcref{fig:FG}).
    The paths $P_u^1$ and $P_u^2$ can then be replaced in the natural way.
    This reduces the length of $P_u^1$.
    Importantly, this construction does not affect $P_w$ or $Q_u$.

    For the other case, let $a,b$ again be the two endpoints of $R_2$ with $b \in V(P_u^1) \setminus \{ u'' \}$.
    We may now let $Q_w'$ be constructed by combining $wQ_wa$ and $R_2$ (see \zcref{fig:FG}).
    The paths $P_u^1$ and $P_u^2$ are again replaced in the natural way.
    This reduces the length of $P_u^1$ without affecting $P_w$ or $Q_u$.

    According to what we observed above, we may repeat this operation until no path fitting the definition of $R_1$ or $R_2$ exists and during this process we will never have to touch $P_w$ or $Q_u$.
    Hence, we can carry out the same procedure on $P_w$ and $Q_u$.
    Using this construction and \zcref{claim:noskipping}, we thus may assume that $P_u$, $P_w$, $Q_u$, and $Q_w$ are chosen such that $S' \coloneqq \{ u'',w'',v' \}$ is a separator of order 3 in $G$.

    Suppose that both $P_u^1$ and $P_w^1$ are non-trivial paths.
    Let $e_u$ be the edge incident to $u''$ that is found in $P_u^1$ and let $e_w$ be the edge incident to $w''$ that is found in $P_w^1$.
    Recall that $e = vv'$.
    We can now further deduce that $F \coloneqq \{ e_u,e_w,e \}$ is a non-trivial cut of order 3.
    
    Towards a contradiction, we will now show that $F$ must in fact be induced.
    If this holds, \zcref{lem:tightcutshape} tells us that $F$ is a tight cut and thus $P_u^1$ or $P_w^1$ must be a trivial path.

    Let $u^\star$ be the endpoint of $e_u$ distinct from $u''$ and let $w^\star$ be the endpoint of $e_w$ distinct from $w''$.
    As all edges incident to $u''$ and $w''$ are found in $E(H)$ and thus $u''w^\star, w''u^\star, vu'', vw'' \not\in E(G)$.
    The following possible edges remain: $u^\star w^\star$, $v'u^\star$, and $v'w^\star$.
    In particular, the existence of the last two edges implies $u^\star = u'$ and respectively, $w^\star = w'$.

    \begin{claim}\label{claim:nobridging}
        We have $u^\star w^\star \not\in E(G)$.
    \end{claim}
    \emph{Proof of \zcref{claim:nobridging}:}
    Suppose that $u^\star w^\star \in E(G)$.
    This immediately implies that $u''$ and $w''$ have different colours in $G$.
    We assume w.l.o.g.\ that $u'' \in A$ and $w'' \in B$.

    Since $G$ is bipartite and $F$ is a cut in $G$, either all perfect matchings of $G$ have an even number of edges in $F$ or they all have an odd number of edges in $F$.
    Let $U$ be the component of $G - F$ that contains $u''$, $w''$, and $v$.
    
    If $|F \cap M|$ is even for every perfect matching $M$, we first dismiss the possibility that $|F \cap M| = 0$.
    Should such a matching exist this would imply that $U$ has a perfect matching and thus as many black as it has white vertices.
    This is contradicted by the fact that, due to the 2-extendability of $G$ there exists a perfect matching $M'$ of $G$ which uses both $e_u$ and $e$, whose endpoints in $U$ are white.
    Thus $|F \cap M| = 2$ for all perfect matchings of $G$.
    But this again runs into parity issues as $M'$ from above tells us that $U$ has a surplus of 2 white vertices, whilst the 2-extendability of $G$ also tells us that $e_u,e_w$ are found together in a perfect matching and thus $U$ again must contain an equal number of white and black vertices.
    
    We can therefore move on to assuming that $|F \cap M|$ is odd for every perfect matchings of $G$.
    Choosing $M'$ as above, we confirm that there exists a perfect matching using the entirety of $F$.
    Thus we have $|V(U) \cap A| = |V(U) \cap B| + 1$.
    We can now count the number of edges in $U$ by counting all edges incident to the black, respectively the white vertices of $U$.
    The first option yields $|E(U)| = 3|V(U) \cap B| - 1$, as $w''$ loses one of its edges to $F$.
    However, as both $u''$ and $v$ lose an edge to $F$, the other option yields
    \[ |E(U)| = 3|V(U) \cap A| - 2  = 3(|V(U) \cap B| + 1) - 2 = 3|V(U) \cap B| + 3 - 2 = 3|V(U) \cap B| + 1, \]
    which contradicts our earlier calculation and proves our claim.
    \hfill$\blacksquare$

    This leaves us to show that $v'u^\star , v'w^\star \not\in E(G)$.
    We first assume that in fact both of these edges are found in $E(G)$, which implies that $\{ u^\star , w^\star \}$ form a 2-separator in $G$, contradicting its 3-connectivity.
    Thus we may assume that only one of these two, say $v'w^\star$, is found in $E(G)$.

    It remains true that $F$ is a non-trivial cut in $G$ and we let $U'$ be the component of $G - F$ that contains $w', v', u', u^\star$.
    Due to $G$ being cubic, there exists a $w'$-$V(P_u^1 - u'')$-path $P$ in $U'$ that is disjoint from $v'$.
    We further know that $P$ cannot have $u'$ as its second endpoint $z$, as otherwise $P \cup P_u \cup P_w \cup G[ \{ u,v,w \} ]$ is a cycle contradicting the cycle-conformality of $G$.
    
    Our goal is to now modify $P_u^1$ until we can position $P$ such that there does not exist a $V(u''P_uz)$-$V(zP_uu')$-path $R$ in $H - v'$ disjoint from $P$.
    If $R$ should exist, we let $a,b$ be the endpoints of $R$ such that $b \in V(zP_uu')$ and modify $P_u^1$ by replacing $aP_ub$ by $R$ and find a new path $P$ which has its endpoint closer to $u'$ than the previous $P$ by choosing $P \cup zP_ub$ (see \zcref{fig:HI}).
    This process must clearly terminate and as we perform all modifications within $U$, we retain that $F$ is a cut separating $P_u^1 - u''$ from the rest of $H - v'$.
    Thus we reach the desired property for $P$.
    However, the fact that $z \neq u'$ and $u'$ having a neighbour outside of $H$ implies that $\{ z,u' \}$ is a 2-separator in $G$, contradicting its 3-connectivity.

    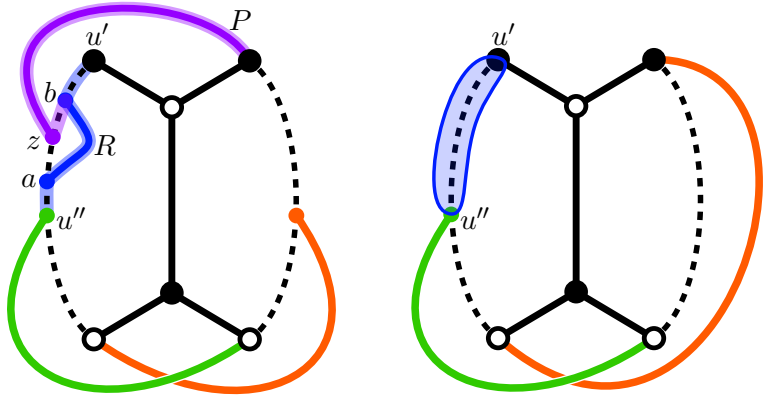
\begin{figure}[ht]
    \centering
        \begin{tikzpicture}[scale=1.25]

            \pgfdeclarelayer{background}
		      \pgfdeclarelayer{foreground}
			
		      \pgfsetlayers{background,main,foreground}

            \begin{pgfonlayer}{background}
            \pgftext{\includegraphics[width=8cm]{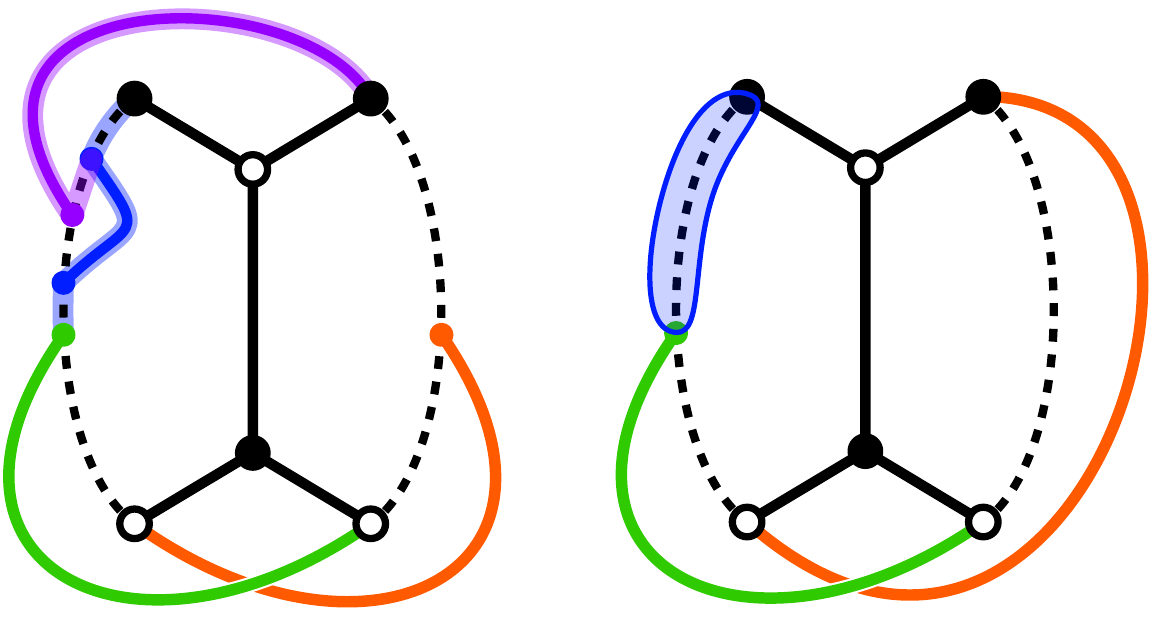}} at (C.center);
            \end{pgfonlayer}{background}
			
            \begin{pgfonlayer}{foreground}

            \node (Z) at (-3.7,0.6) [draw=none] {$z$};
            \node (A) at (-3.75,0.2) [draw=none] {$a$};
            \node (B) at (-3.525,1.1375) [draw=none] {$b$};
            \node (UPRIME) at (-3.025,1.75) [draw=none] {$u'$};
            \node (UPRIMEPRIME) at (-3.3,-0.2) [draw=none] {$u''$};
            \node (P) at (-1.525,1.9) [draw=none] {$P$};
            \node (R) at (-2.95,0.575) [draw=none] {$R$};

            \node (UPRIMEAGAIN) at (1.25,1.75) [draw=none] {$u'$};
            \node (UPRIMEPRIMEAGAIN) at (0.95,-0.2) [draw=none] {$u''$};
            
            \end{pgfonlayer}{foreground}
        
        \end{tikzpicture}
    \caption{Illustrations for the last steps in the proof of \zcref{lem:k33subdivision}. On the left the process of moving $P$ along $P^1_u$ is illustrated, with the purple highlighted path representing the updated path $P$ and the path highlighted in blue represented the new path $P^1_u$. On the right a 2-separator in $G$ is indicated in the case in which $P_w^1$ is trivial but $P_u^1$ is not.}
    \label{fig:HI}
    \end{figure}

    We may thus conclude that at least one path amongst $P_u^1$ and $P_w^1$ is trivial.
    Suppose w.l.o.g.\ that $P_u^1$ is not trivial.
    Then it is easy to observe that $u'$ and $u''$ form a 2-separator in $G$ (see \zcref{fig:HI}).
    Hence in fact both $P_u^1$ and $P_w^1$ are trivial.
    This allows us to drop the indices and see $H$ as being comprised of $v$, $v'$, $P_u$, $P_w$, $Q_u$, and $Q_w$.
    In particular, $H$ is therefore a subdivision of $K_{3,3}$, as desired.
\end{proof}

We now use this lemma as a starting point to finish the proof of the main theorem in this section.

\begin{proof}[Proof of \zcref{thm:k33issolo}]
    Suppose that $G$ is not isomorphic to $K_{3,3}$.
    Then, according to \zcref{lem:k33subdivision}, $G$ contains a $K_{3,3}$-subdivision $H$ containing five edges $e = vv', uv, vw, u'v', v'w'$ such that $u,w,u',w'$ each have degree 3 in $H$.
    This in particular implies that there exists a unique $u$-$u'$-path $P_u$, a unique $w$-$w'$-path, a unique $u$-$w'$-path $Q_w$, and a unique $w$-$u'$-path in $H' \coloneqq H - \{ v,v' \}$, such that all of these paths are internally disjoint and we have $\{ v,v' \} \cup V(P_u \cup P_w \cup Q_u \cup Q_w) = V(H)$.
    
    Of course, if each of these paths has length 1, then $G$ is isomorphic to $K_{3,3}$.
    Thus there must exist a vertex on at least one of our paths that is not from $\{ u,u',w,w' \}$.
    We first discuss that this vertex cannot lead us to any other part of our structure.
    A claim very similar to \zcref{claim:noskipping} does most of the work.

    \begin{claim}\label{claim:noskipping2}
        The following do not exist:
        \begin{enumerate}
            \item A $V(P_u)$-$V(P_w)$-path in $G'$ that is internally disjoint from $H$.

            \item A $V(Q_u)$-$V(Q_w)$-path in $G'$ that is internally disjoint from $H$.
        \end{enumerate}
    \end{claim}
    \emph{Proof of \zcref{claim:noskipping2}:}
    This claim is easy to confirm by hand by finding non-conformal cycles (see \zcref{fig:JK} for an illustration).
    \hfill$\blacksquare$

    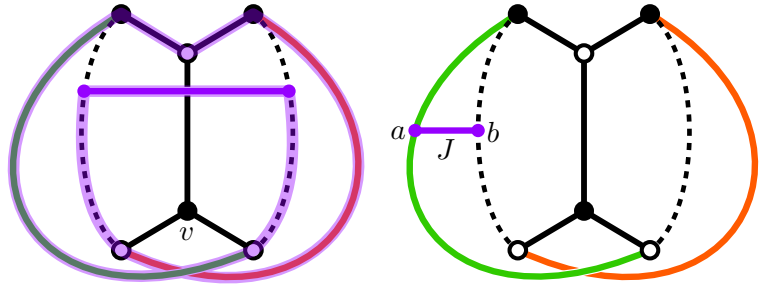
\begin{figure}[ht]
    \centering
        \begin{tikzpicture}[scale=1.25]

            \pgfdeclarelayer{background}
		      \pgfdeclarelayer{foreground}
			
		      \pgfsetlayers{background,main,foreground}

            \begin{pgfonlayer}{background}
            \pgftext{\includegraphics[width=8cm]{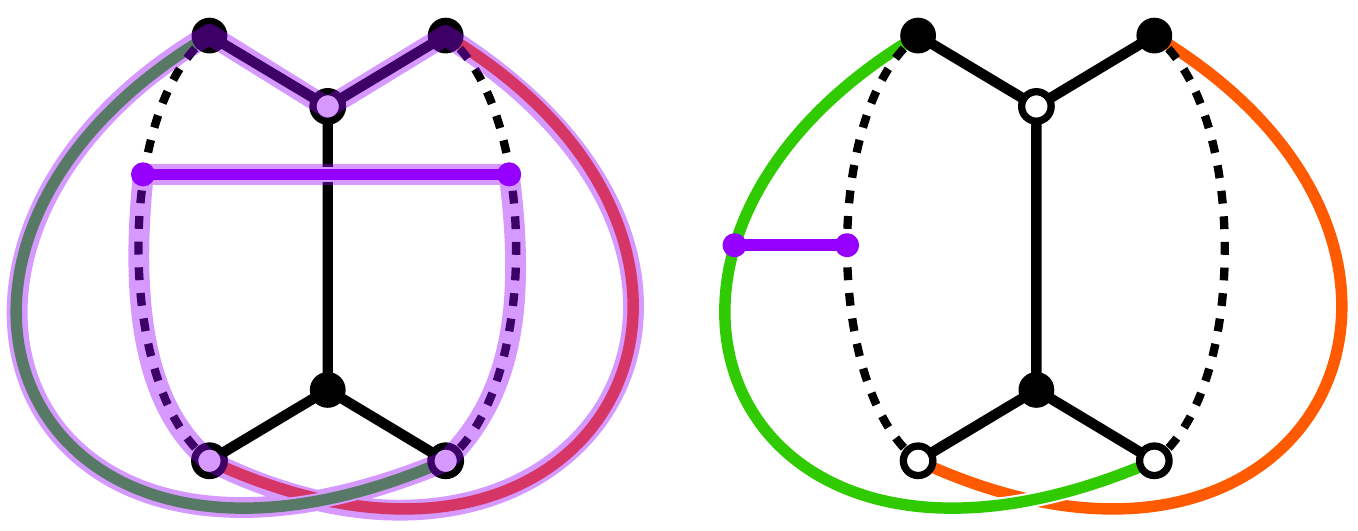}} at (C.center);
            \end{pgfonlayer}{background}
			
            \begin{pgfonlayer}{foreground}
            
            \node (V) at (-2.075,-0.975) [draw=none] {$v$};

            \node (J) at (0.65,-0.075) [draw=none] {$J$};
            \node (A) at (0.15,0.1) [draw=none] {$a$};
            \node (B) at (1.15,0.1) [draw=none] {$b$};
            
            \end{pgfonlayer}{foreground}
        
        \end{tikzpicture}
    \caption{The illustration to the left shows why \zcref{claim:noskipping2} holds.
    Here a $V(P_u)$-$V(P_w)$-path is indicated in purple and a cycle isolating the vertex $v$ is highlighted in purple, which confirms that in this case $G$ cannot be cycle-conformal.
    The other case of \zcref{claim:noskipping2} is symmetric.
    On the right we illustrate where the path $J$ lies.}
    \label{fig:JK}
    \end{figure}

    At this point the endpoints of $Q_u$ are $u$ and $w'$, and the endpoints of $Q_w$ are $w$ and $u'$.
    This allows us to relabel $H$ such that up to symmetry there is now only one remaining case to consider.
    We therefore assume w.l.o.g.\ that there exists a non-trivial $V(Q_w)$-$V(P_u)$-path $J$ in $G$ that is internally disjoint from $H$.
    In particular, we know that $J$ is disjoint from $u$, $u'$, and $w$.
    Let $a \in V(P_u)$ and $b \in V(Q_w)$ be the two endpoints of $J$ (see \zcref{fig:JK}).

    Further, let $K'$ be the component of $G - V(H)$ that contains the internal vertices of $J$.
    This may mean that $K'$ is empty.
    Since $K'$ has neighbours in $V(P_u - \{ u,u' \})$ and $V(Q_w - \{ u',w \})$ it cannot have any neighbours in $V(P_w)$ or $V(Q_u)$ according to \zcref{claim:noskipping2}.
    This substantiates part of the following claim. 
    
    \begin{claim}\label{claim:notriangleescape}
        There do not exist any $V(J \cup aP_uu' \cup u'Q_wb)$-$V(P_w \cup Q_u)$-paths internally disjoint from $H$.
    \end{claim}
    \emph{Proof of \zcref{claim:notriangleescape}:}
    This claim is easy to confirm by hand by finding non-conformal cycles  (see \zcref{fig:L} for an illustration).
    \hfill$\blacksquare$

    \begin{figure}[ht]
    \centering
        \begin{tikzpicture}[scale=1.25]

            \pgfdeclarelayer{background}
		      \pgfdeclarelayer{foreground}
			
		      \pgfsetlayers{background,main,foreground}

            \begin{pgfonlayer}{background}
            \pgftext{\includegraphics[width=8cm]{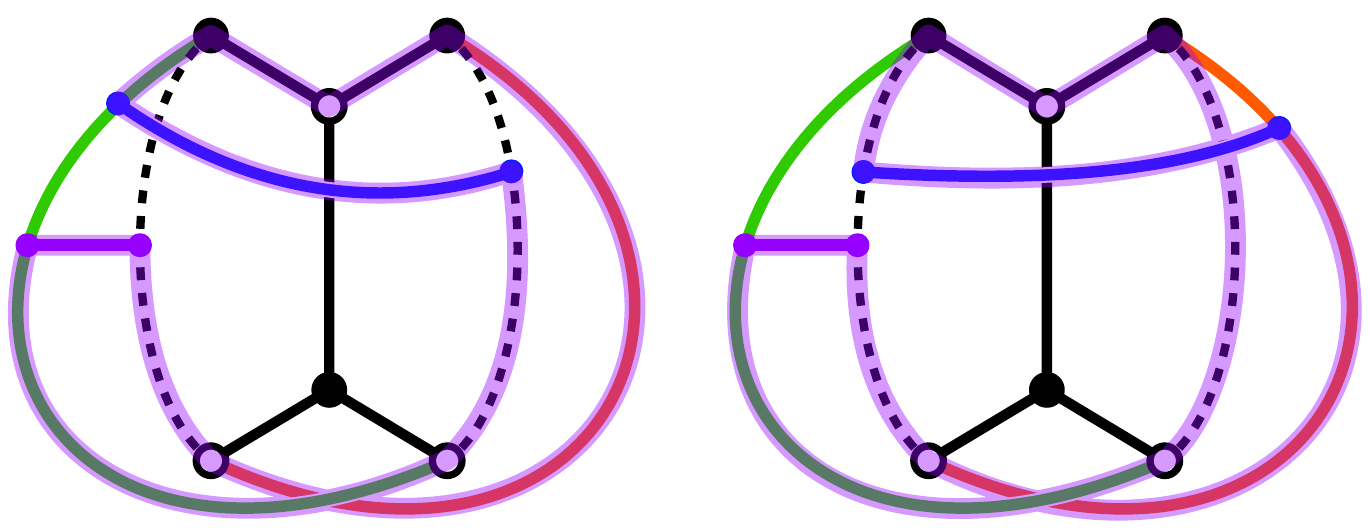}} at (C.center);
            \end{pgfonlayer}{background}
			
            \begin{pgfonlayer}{foreground}

            \node (V) at (-2.1,-1) [draw=none] {$v$};
            \node (VAGAIN) at (2.1,-1) [draw=none] {$v$};
            
            \end{pgfonlayer}{foreground}
        
        \end{tikzpicture}
    \caption{An illustration for \zcref{claim:notriangleescape}. The blue paths represent the type of path the claim forbids from existing. In both figures the cycle highlighted in purple isolates $v$ and thus contradicts the cycle-conformality of $G$.}
    \label{fig:L}
    \end{figure}
    
    This claim together with \zcref{claim:noskipping2} tells us that there are two types of paths over which we do not yet have control.
    First, it is possible to have a $V(aP_uu)$-$V(Q_u)$-path that is internally disjoint from $H$, with $u$ itself being such a path.
    Secondly, by symmetry, a $V(bQ_ww)$-$V(P_w)$-path internally disjoint from $H$ may exist, with $w$ being such a path.
    Both of these types of paths must be disjoint from $V(K')$.

    Let $a'$ be the vertex in $V(aP_uu)$ that minimises the length of $aP_ua'$ such that there exists an $a'$-$V(Q_u)$-path in $G$ that is internally disjoint from $H$.
    We let $b'$ be defined analogously for $bQ_ww$.
    Note that $a \neq a'$ and $b \neq b'$, as $G$ is cubic.
    Finally, let $f_a$ be the unique edge in $aP_ua'$ that is incident to $a'$ and let $f_b$ be defined analogously for $bQ_wb'$.

    We claim that $F' = \{ uv', f_a, f_b \}$ is a tight cut.
    By construction it is easy to see that $F'$ is indeed a cut, since according to \zcref{claim:notriangleescape}, we cannot escape $V(J \cup aP_uu' \cup u'Q_wb)$ and by choice of $a'$ and $b'$, it is also not possible to escape $V(aP_ua')$ and $V(bQ_wb')$ into another part of $H$.
    
    It remains to confirm that $F'$ is an induced matching of order 3, allowing us to apply \zcref{lem:tightcutshape} to conclude our proof.
    That $F'$ is a matching is easy to see.
    Both $f_a$ and $f_b$ cannot have endpoints incident to endpoints of $uv'$ due to the existence of $J$ and the general structure of $H$.
    Two of the edges incident to $a'$ are contained in $E(P_u)$ and the third is found in an $a'$-$V(Q_u)$-path and thus $a'$ is safe as well.
    An analogous argument frees up $b'$.
    If the remaining endpoints $a^\star$ and $b^\star$ from $aP_ua'$ and $bQ_wb'$ respectively are adjacent, then the path $a^\star P_uu \cup uQ_wb^\star$ can be closed via the edge $a^\star b^\star$.
    Since $G$ is bipartite, this means that $a^\star$ and $b^\star$ have distinct colours.
    To conclude that this is impossible, we may simply repeat the arguments presented in the proof of \zcref{claim:nobridging}.
    This completes our proof.
\end{proof}

Combining \zcref{lem:tightcutpreserve}, \zcref{thm:cubictightcuts}, and \zcref{thm:k33issolo} immediately yields \zcref{thm:cubiccharacterisation}.

\section{Conclusion}\label{sec:generation}
In this article we showed that $C_4$ is the only Pfaffian, cycle-conformal brace and $K_{3,3}$ is the only cubic, cycle-conformal brace.
Both of these results were then extended to characterise the matching covered, Pfaffian, bipartite, cycle-conformal graphs and respectively, characterise the matching covered, cubic, bipartite, cycle-conformal graphs.
We close with a discussion of whether some select fundamental bricks and braces are cycle-conformal and a direction for future research.

\paragraph{Interesting classes of bricks and braces.}
There are a few prominent classes of bricks and braces in matching theory for which it is natural to ask whether they are cycle-conformal.
In~\cite{DalwadiPDK2025Planar}, the planar representatives among the so-called Norine-Thomas bricks (named after their appearance in~\cite{NorineT2007Generating}) are discussed.
Let $C$ be an odd cycle and let $G$ be the result of adding a single vertex adjacent to all vertices in $V(C)$ to $C$.
We call $G$ an \emph{odd wheel}.
Let $k \in \mathbb{N}$ be positive, then the \emph{ladder $L_k$ with $k$ rungs} is defined as $(\{ u_i,v_i \mid i \in [k] \}, \{ u_iu_{i+1}, v_iv_{i+1}, u_iv_i \mid i \in [k-1] \} \cup \{ u_kv_k \})$.
An \emph{odd prism} is constructed from a ladder with $k \geq 3$ rungs, with $k$ being odd, by adding the edges $u_1u_k, v_1v_k$.
The odd wheels and the odd prisms represent all planar, cycle-conformal Norine-Thomas bricks~\cite{DalwadiPDK2025Planar}.

The importance of the Norine-Thomas bricks derives from the fact that they represent a fundamental class of bricks from which all bricks can be generated via fairly simple operations.
A similar approach was important to McCuaig's resolution of \Polya's Permanent Problem~\cite{McCuaig2001Brace,McCuaig2004Polyas}.
For this reason, we now briefly address the non-planar Norine-Thomas bricks.

\begin{figure}[ht]
     \centering
            \begin{tikzpicture}
			
			\node (V0) at (0:0) [draw=none] {};
			
			\foreach\i in {1,...,5}
			{
				\node (V\i) at ($(V0)+({(360/5 * \i)+360/4}:1)$) [draw, circle, scale=0.6, fill, label={}] {};
			}
			
			\foreach\i in {6,...,10}
			{
				\node (V\i) at ($(V0)+({(360/5 * \i)+360/4}:2)$) [draw, circle, scale=0.6, fill, label={}] {};
			}
			
			\foreach\i in {6,7,8,9}
			{
				\pgfmathtruncatemacro\iplus{\i+1}
				\path (V\i) edge[very thick] (V\iplus);
			}
			
			\path
			(V1) edge[very thick] (V6)
                (V2) edge[very thick] (V7)
                (V3) edge[very thick] (V8)
                (V4) edge[very thick] (V9)
                (V5) edge[very thick] (V10)
                (V1) edge[very thick] (V3)
                (V1) edge[very thick] (V4)
                (V2) edge[very thick] (V4)
                (V2) edge[very thick] (V5)
                (V3) edge[very thick] (V5)
                (V6) edge[very thick] (V10)
			;
			
			\end{tikzpicture}
     \qquad
            \begin{tikzpicture}
			
			\node (V0) at (0:0) [draw=none] {};
			
			\foreach\i in {1,...,5}
			{
				\node (V\i) at ($(V0)+({(360/5 * \i)+360/4}:1)$) [draw, circle, scale=0.6, fill, label={}] {};
			}
			
			\foreach\i in {6,...,10}
			{
				\node (V\i) at ($(V0)+({(360/5 * \i)+360/4}:2)$) [draw, circle, scale=0.6, fill, label={}] {};
			}
			
			\foreach\i in {6,7,8,9}
			{
				\pgfmathtruncatemacro\iplus{\i+1}
				\path (V\i) edge[very thick] (V\iplus);
			}
			
			\path
			(V1) edge[very thick] (V6)
                (V2) edge[very thick] (V7)
                (V3) edge[very thick,red] (V8)
                (V4) edge[very thick,red] (V9)
                (V5) edge[very thick] (V10)
                (V1) edge[very thick] (V3)
                (V1) edge[very thick] (V4)
                (V2) edge[very thick] (V4)
                (V2) edge[very thick] (V5)
                (V3) edge[very thick] (V5)
                (V6) edge[very thick] (V10)
                (V5) edge[very thick,red] (V2)
                (V2) edge[very thick,red] (V4)
                (V9) edge[very thick,red] (V8)
                (V3) edge[very thick,red] (V5)
			;
			
			\end{tikzpicture}
     \caption{Two drawings of the Petersen graph. On the right a non-conformal, even cycle in the Petersen graph is marked in red.}
     \label{fig:petersen}
\end{figure}
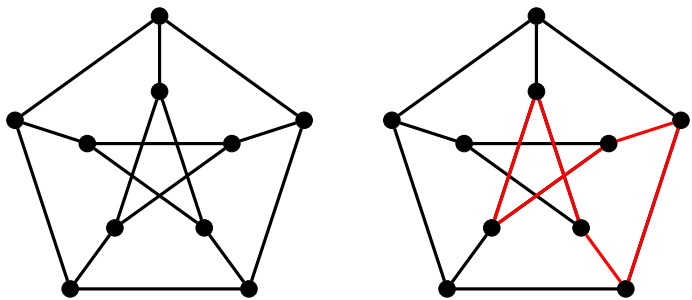

Let $k \in \mathbb{N}$ with $k \geq 3$, then the graph $M_k$ which is derived from $L_k$ by adding the edges $u_1v_k$ and $v_1u_k$ is called a \emph{Möbius ladder with $k$ rungs}.
We call a Möbius ladder $M_k$ \emph{even} if $k$ is even and \emph{odd} if $k$ is odd.
The non-planar Norine-Thomas bricks then consist of even Möbius ladders and the Petersen graph~\cite{NorineT2007Generating}.
While for the latter it can easily be observed that it is not cycle-conformal (see \zcref{fig:petersen}), the Möbius ladders are in fact cycle-conformal.

\begin{lemma}\label{lem:moebiscycconf}
    Even Möbius ladders are cycle-conformal.
\end{lemma}
\begin{proof}
    It is easy to observe that ladders themselves are cycle-conformal.
    To confirm that the same holds true for even Möbius ladders one therefore only needs to consider even cycles which use at least one of the two newly added edges.
    Next, we observe that ladders are bipartite and each of the two added edges in an even Möbius ladder connects two vertices in the same colour class.
    Hence a cycle using the new edges can be even if and only if it uses both of these edges.
    It is now easy to observe that unless such a cycle spans the entire graph (in which case it is conformal), its deletion leaves us with a ladder, whose rungs form a perfect matching.
\end{proof}

In his paper on the generation of braces~\cite{McCuaig2001Brace}, McCuaig presents three fundamental classes of braces, two of which are planar, with each member being non-isomorphic to $C_4$ and thus being non-cycle-conformal.
The remaining class are the odd Möbius ladders.
Notably $M_3$ is isomorphic to $K_{3,3}$ and thus cycle-conformal.
However, all other odd Möbius ladders are not cycle-conformal (see \zcref{fig:moebius} or use the fact that they are cubic combined with \zcref{thm:k33issolo}).
Thus there is exactly one graph among the basic braces in~\cite{McCuaig2001Brace} that is cycle-conformal.

\begin{figure}[ht]
        \centering
            \begin{tikzpicture}

            \clip (-0.5,-0.25) rectangle (6.5,2.6);

            \node (A0) at (0,0) [draw, circle, scale=0.6] {};
            \node (A1) at (1.5,1.5) [draw, circle, scale=0.6] {};
            \node (A2) at (3,0) [draw, circle, scale=0.6] {};
            \node (A3) at (4.5,1.5) [draw, circle, scale=0.6] {};
            \node (A4) at (6,0) [draw, circle, scale=0.6] {};
            \node (B0) at (0,1.5) [draw, circle, scale=0.6, fill] {};
            \node (B1) at (1.5,0) [draw, circle, scale=0.6, fill] {};
            \node (B2) at (3,1.5) [draw, circle, scale=0.6, fill] {};
            \node (B3) at (4.5,0) [draw, circle, scale=0.6, fill] {};
            \node (B4) at (6,1.5) [draw, circle, scale=0.6, fill] {};

            \path
                (A0) edge[very thick,red] (B0)
                (A0) edge[very thick,red] (B1)
                (A1) edge[very thick] (B0)
                (A1) edge[very thick,red] (B1)
                (A1) edge[very thick,red] (B2)
                (A2) edge[very thick] (B1)
                (A2) edge[very thick] (B2)
                (A2) edge[very thick] (B3)
                (A3) edge[very thick,red] (B2)
                (A3) edge[very thick,red] (B3)
                (A3) edge[very thick] (B4)
                (A4) edge[very thick,red] (B3)
                (A4) edge[very thick] (B4)
            ;

            \node (Con2) at (-1.5,4) [draw=none, label={}] {};
            \node (Con3) at (-0.5,0.25) [draw=none, label={}] {};

            \draw[very thick] (B4) .. controls (Con2) and (Con3) .. (A0);

            \node (Con2) at (6.5,0.25) [draw=none, label={}] {};
            \node (Con3) at (7.5,4) [draw=none, label={}] {};

            \draw[very thick,red] (A4) .. controls (Con2) and (Con3) .. (B0);

		\end{tikzpicture}
     \caption{A drawing of the Möbius ladder $M_5$ with a non-conformal, even cycle indicated in red. The construction of this even cycle can clearly be extended to any odd Möbius ladder $M_k$ with $k \geq 3$, rendering all of these graphs non-cycle-conformal.}
     \label{fig:moebius}
\end{figure}
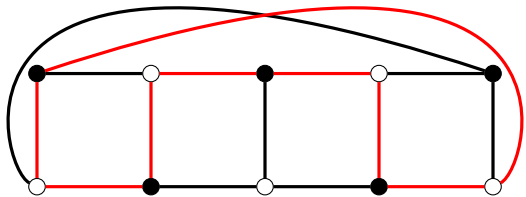

\paragraph{Odd-cycle-conformal graphs.}
Recall that in the introduction we made the decision to focus on graphs in which all even cycles are conformal to define cycle-conformal graphs.
As laid out throughout this paper, there are many good reasons to make this choice.
There is however a sensible class for which it is interesting to consider the graphs in which all odd cycles are conformal, which we call \emph{odd-cycle-conformal}.

A graph $G$ is called factor-critical if for all vertices $v \in V(G)$ the graph $G - v$ has a perfect matching.
Such graphs must clearly have an odd number of vertices and thus, for any odd cycle $C \subseteq G$ the set $V(G) \setminus V(C)$ has an even number of vertices.
Furthermore, it is easy to see that a factor-critical graph $G$ cannot be a forest, meaning that $G$ must contain a cycle, and $G$ cannot be bipartite.
These observations, combined with their fundamental role in matching theory (see~\cite{LovaszP1986Matching} for reference), indicate that factor-critical graphs are the appropriate target for the study of odd-cycle-conformal graphs.

As for examples of odd-cycle-conformal graphs, it is easy to observe that for all integers $t$ the cycle $C_{2t+1}$ is odd-cycle-conformal and the same is true for $K_{2t+1}$.
An interesting, non-trivial planar example of a class of odd-cycle-conformal graphs is given by the even wheels, which are defined analogously to odd wheels.
Note that all of these examples are factor-critical.

In general, over the course of history in matching theory, factor-critical graphs have proven to be easier to deal with than matching covered graphs.
Thus, in contrast to the apparent complexity of the task for cycle-conformal, matching covered graphs, a full characterisation of the factor-critical, odd-cycle-conformal graphs may not be out of reach.

\medskip

\textbf{Acknowledgements:} Parts of this article, mainly \zcref{sec:braces}, are based on the Bachelor thesis of the second author~\cite{Kuske2024Characterization}, which was supervised by the first author.

We also want to thank Sebastian Wiederrecht for suggesting to us that we study cycle-conformal graphs, a topic which he first developed together with Nishad Kothari whilst on a visit to the University of Waterloo in 2017.

\bibliographystyle{alphaurl}
\bibliography{literature}

@article{CheC2008Forcing,
  title = {Forcing Faces in Plane Bipartite Graphs},
  author = {Che, Zhongyuan and Chen, Zhibo},
  year = 2008,
  month = jun,
  journal = {Discrete Mathematics},
  volume = {308},
  number = {12},
  pages = {2427--2439},
  issn = {0012365X},
  doi = {10.1016/j.disc.2007.05.025},
  urldate = {2023-05-01},
  langid = {english}
}

@article{CheC2013Forcing,
  title = {Forcing Faces in Plane Bipartite Graphs ({{II}})},
  author = {Che, Zhongyuan and Chen, Zhibo},
  year = 2013,
  month = jan,
  journal = {Discrete Applied Mathematics},
  volume = {161},
  number = {1-2},
  pages = {71--80},
  issn = {0166218X},
  doi = {10.1016/j.dam.2012.08.016},
  urldate = {2023-05-01},
  langid = {english}
}

@article{DalwadiPDK2025Planar,
  title = {Planar Cycle-Extendable Graphs},
  author = {Dalwadi, Aditya Y and Pause, Kapil R Shenvi and Diwan, Ajit A and Kothari, Nishad},
  year = 2025,
  month = may,
  journal = {Discrete Mathematics \& Theoretical Computer Science},
  volume = {vol. 27:2},
  number = {Graph Theory},
  pages = {13929},
  issn = {1365-8050},
  doi = {10.46298/dmtcs.13929},
  urldate = {2025-06-02},
  langid = {english}
}

@article{DeCarvalhoC2005An,
  title = {An $\mathcal{O}{ ( VE )}$ algorithm for ear decompositions of matching-covered graphs},
  author = {de Carvalho, Marcelo H and Cheriyan, Joseph},
  year = 2005,
  month = oct,
  journal = {ACM Transactions on Algorithms},
  volume = {1},
  number = {2},
  pages = {324-337},
  issn = {1549-6325, 1549-6333},
  doi = {10.1145/1103963.1103969},
  urldate = {2023-05-04},
  langid = {english}
}

@article{deCarvalhoKWL2020Birkhoffvon,
  title = {Birkhoff--von {{Neumann Graphs}} That Are {{PM-Compact}}},
  author = {{de Carvalho}, Marcelo H. and Kothari, Nishad and Wang, Xiumei and Lin, Yixun},
  year = 2020,
  month = jan,
  journal = {SIAM Journal on Discrete Mathematics},
  volume = {34},
  number = {3},
  pages = {1769--1790},
  issn = {0895-4801, 1095-7146},
  doi = {10.1137/18M1202347},
  urldate = {2025-06-18},
  langid = {english}
}

@book{Diestel2010Graph,
  title = {Graph Theory},
  author = {Diestel, Reinhard},
  year = 2010,
  series = {Graduate Texts in Mathematics},
  edition = {4th ed},
  number = {173},
  publisher = {Springer},
  address = {Heidelberg ; New York},
  isbn = {978-3-642-14278-9 978-3-642-14279-6},
  langid = {english},
  lccn = {QA166 .D51413 2010}
}

@article{FischerL2001Characterisation,
  title = {A {{Characterisation}} of {{Pfaffian Near Bipartite Graphs}}},
  author = {Fischer, Ilse and Little, Charles H C},
  year = 2001,
  month = jul,
  journal = {Journal of Combinatorial Theory, Series B},
  volume = {82},
  number = {2},
  pages = {175--222},
  issn = {00958956},
  doi = {10.1006/jctb.2000.2025},
  urldate = {2023-05-02},
  langid = {english}
}

@article{GiannopoulouKW2024Excluding,
  title = {Excluding a Planar Matching Minor in Bipartite Graphs},
  author = {Giannopoulou, Archontia C and Kreutzer, Stephan and Wiederrecht, Sebastian},
  year = 2024,
  month = jan,
  journal = {Journal of Combinatorial Theory, Series B},
  volume = {164},
  pages = {161--221},
  issn = {00958956},
  doi = {10.1016/j.jctb.2023.09.003},
  urldate = {2024-03-13},
  langid = {english}
}

@inproceedings{GiannopoulouTW2023Excluding,
  title = {Excluding {{Single-Crossing Matching Minors}} in {{Bipartite Graphs}}},
  booktitle = {Proceedings of the 2023 {{Annual ACM-SIAM Symposium}} on {{Discrete Algorithms}} ({{SODA}})},
  author = {Giannopoulou, Archontia C and Thilikos, Dimitrios M and Wiederrecht, Sebastian},
  year = 2023,
  month = jan,
  pages = {2111--2121},
  publisher = {{Society for Industrial and Applied Mathematics}},
  address = {Philadelphia, PA},
  doi = {10.1137/1.9781611977554.ch81},
  urldate = {2024-03-18},
  langid = {english}
}

@misc{GiannopoulouW2021Two,
  title = {Two {{Disjoint Alternating Paths}} in {{Bipartite Graphs}}},
  author = {Giannopoulou, Archontia C and Wiederrecht, Sebastian},
  year = 2021,
  month = oct,
  number = {arXiv:2110.02013},
  eprint = {2110.02013},
  primaryclass = {cs, math},
  publisher = {arXiv},
  urldate = {2023-05-01},
  archiveprefix = {arXiv}
}

@inproceedings{GiannopoulouW2024Flat,
  title = {A {{Flat Wall Theorem}} for {{Matching Minors}} in {{Bipartite Graphs}}},
  booktitle = {Proceedings of the 56th {{Annual ACM Symposium}} on {{Theory}} of {{Computing}}},
  author = {Giannopoulou, Archontia C. and Wiederrecht, Sebastian},
  year = 2024,
  month = jun,
  pages = {716--727},
  publisher = {ACM},
  address = {Vancouver BC Canada},
  doi = {10.1145/3618260.3649774},
  urldate = {2025-06-18},
  isbn = {979-8-4007-0383-6},
  langid = {english}
}

@phdthesis{Gorsky2024Structure,
  title = {The Structure of (Even) Directed Cycles},
  author = {Gorsky, Maximilian},
  year = 2024,
  month = sep,
  doi = {10.14279/depositonce-21276},
  urldate = {2024-09-17},
  copyright = {Creative Commons Attribution 4.0 International},
  langid = {english},
  school = {Technische Universit\"at Berlin}
}

@inproceedings{GorskyKKW2024Packing,
  title = {Packing {{Even Directed Circuits Quarter-Integrally}}},
  booktitle = {Proceedings of the 56th {{Annual ACM Symposium}} on {{Theory}} of {{Computing}}},
  author = {Gorsky, Maximilian and Kawarabayashi, Ken-ichi and Kreutzer, Stephan and Wiederrecht, Sebastian},
  year = 2024,
  month = jun,
  pages = {692--703},
  publisher = {ACM},
  address = {Vancouver BC Canada},
  doi = {10.1145/3618260.3649682},
  urldate = {2024-06-12},
  isbn = {979-8-4007-0383-6},
  langid = {english}
}

@article{GuoZ2004Reducible,
  title = {Reducible Chains of Planar 1-Cycle Resonant Graphs},
  author = {Guo, Xiaofeng and Zhang, Fuji},
  year = 2004,
  month = jan,
  journal = {Discrete Mathematics},
  volume = {275},
  number = {1-3},
  pages = {151--164},
  issn = {0012365X},
  doi = {10.1016/S0012-365X(03)00102-X},
  urldate = {2025-01-08},
  copyright = {https://www.elsevier.com/tdm/userlicense/1.0/},
  langid = {english}
}

@article{Hendry1990Extending,
  title = {Extending Cycles in Graphs},
  author = {Hendry, George R.T.},
  year = 1990,
  month = nov,
  journal = {Discrete Mathematics},
  volume = {85},
  number = {1},
  pages = {59--72},
  issn = {0012365X},
  doi = {10.1016/0012-365X(90)90163-C},
  urldate = {2025-06-02},
  copyright = {https://www.elsevier.com/tdm/userlicense/1.0/},
  langid = {english}
}

@article{Kasteleyn1963Dimer,
  title = {Dimer {{Statistics}} and {{Phase Transitions}}},
  author = {Kasteleyn, Pieter W},
  year = 1963,
  month = feb,
  journal = {Journal of Mathematical Physics},
  volume = {4},
  number = {2},
  pages = {287--293},
  issn = {0022-2488, 1089-7658},
  doi = {10.1063/1.1703953},
  urldate = {2023-05-03},
  langid = {english}
}

@article{KlavzarS2012Characterization,
  title = {A Characterization of 1-Cycle Resonant Graphs among Bipartite 2-Connected Plane Graphs},
  author = {Klav{\v z}ar, Sandi and Salem, Khaled},
  year = 2012,
  month = may,
  journal = {Discrete Applied Mathematics},
  volume = {160},
  number = {7-8},
  pages = {1277--1280},
  issn = {0166218X},
  doi = {10.1016/j.dam.2011.12.030},
  urldate = {2025-01-08},
  copyright = {https://www.elsevier.com/tdm/userlicense/1.0/},
  langid = {english}
}

@article{Kuratowski1930Probleme,
  title = {Sur Le Probl\`eme Des Courbes Gauches En {{Topologie}}},
  author = {Kuratowski, Casimir},
  year = 1930,
  journal = {Fundamenta Mathematicae},
  volume = {15},
  pages = {271--283},
  issn = {0016-2736, 1730-6329},
  doi = {10.4064/fm-15-1-271-283},
  urldate = {2024-04-29},
  langid = {english}
}

@phdthesis{Kuske2024Characterization,
  type = {Bachelor {{Thesis}}},
  title = {Towards a Characterization of Cycle-Conformal Graphs},
  author = {Kuske, Clemens},
  year = 2024,
  month = mar,
  school = {Technische Universit\"at Berlin}
}

@article{LaPaughP1984Evenpath,
  title = {The Even-Path Problem for Graphs and Digraphs},
  author = {LaPaugh, Andrea S and Papadimitriou, Christos H},
  year = 1984,
  journal = {Networks},
  volume = {14},
  number = {4},
  pages = {507--513},
  issn = {00283045, 10970037},
  doi = {10.1002/net.3230140403},
  urldate = {2023-05-03},
  langid = {english}
}

@article{LiuW2014Note,
  title = {A Note on {{PM-compact}} Bipartite Graphs},
  author = {Liu, Jinfeng and Wang, Xiumei},
  year = 2014,
  journal = {Discussiones Mathematicae Graph Theory},
  volume = {34},
  number = {2},
  pages = {409},
  issn = {1234-3099, 2083-5892},
  doi = {10.7151/dmgt.1706},
  urldate = {2025-06-18},
  langid = {english}
}

@article{Lovasz1987Matching,
  title = {Matching Structure and the Matching Lattice},
  author = {Lov{\'a}sz, L{\'a}szl{\'o}},
  year = 1987,
  month = oct,
  journal = {Journal of Combinatorial Theory, Series B},
  volume = {43},
  number = {2},
  pages = {187--222},
  issn = {00958956},
  doi = {10.1016/0095-8956(87)90021-9},
  urldate = {2023-05-01},
  langid = {english}
}

@book{LovaszP1986Matching,
  title = {Matching {{Theory}}},
  author = {Lov{\'a}sz, L{\'a}szl{\'o} and Plummer, Michael D},
  year = 1986,
  volume = {121},
  publisher = {Elsevier Science Publishers B.V.}
}

@book{LovaszP2009Matching,
  title = {Matching {{Theory}}},
  author = {Lov{\'a}sz, L{\'a}szl{\'o} and Plummer, Michael D},
  year = 2009,
  publisher = {AMS Chelsea Pub},
  address = {Providence, R.I},
  isbn = {978-0-8218-4759-6},
  lccn = {QA164 .L7 2009}
}

@book{LucchesiM2024Perfect,
  title = {Perfect Matchings: A Theory of Matching Covered Graphs},
  shorttitle = {Perfect Matchings},
  author = {Lucchesi, Cl{\'a}udio L. and Murty, U. S. R.},
  year = 2024,
  series = {Algorithms and Computation in Mathematics},
  number = {volume 31},
  publisher = {Springer},
  address = {Cham},
  isbn = {978-3-031-47503-0},
  langid = {english}
}

@article{McCuaig2000Even,
  title = {Even {{Dicycles}}},
  author = {McCuaig, William D},
  year = 2000,
  month = aug,
  journal = {Journal of Graph Theory},
  volume = {35},
  number = {1},
  pages = {46--68},
  doi = {10.1002/1097-0118(200009)35:1<46::AID-JGT4>3.0.CO;2-W}
}

@article{McCuaig2001Brace,
  title = {Brace Generation},
  author = {McCuaig, William D},
  year = 2001,
  month = nov,
  journal = {Journal of Graph Theory},
  volume = {38},
  number = {3},
  pages = {124--169},
  issn = {0364-9024, 1097-0118},
  doi = {10.1002/jgt.1029},
  urldate = {2023-05-01},
  langid = {english}
}

@article{McCuaig2004Polyas,
  title = {P\'olya's {{Permanent Problem}}},
  author = {McCuaig, William D},
  year = 2004,
  month = nov,
  journal = {The Electronic Journal of Combinatorics},
  volume = {11},
  number = {1},
  pages = {R79},
  issn = {1077-8926},
  doi = {10.37236/1832},
  urldate = {2023-05-01}
}

@article{Menger1927Zur,
  title = {Zur Allgemeinen {{Kurventheorie}}},
  author = {Menger, Karl},
  year = 1927,
  journal = {Fundamenta Mathematicae},
  volume = {10},
  number = {1},
  pages = {96--115},
  issn = {0016-2736},
  doi = {10.4064/fm-10-1-96-115}
}

@article{NorineT2007Generating,
  title = {Generating Bricks},
  author = {Norine, Serguei and Thomas, Robin},
  year = 2007,
  month = sep,
  journal = {Journal of Combinatorial Theory, Series B},
  volume = {97},
  number = {5},
  pages = {769--817},
  issn = {00958956},
  doi = {10.1016/j.jctb.2007.01.002},
  urldate = {2025-06-18},
  copyright = {https://www.elsevier.com/tdm/userlicense/1.0/},
  langid = {english}
}

@mastersthesis{Pause2022Planar,
  title = {On Planar Cycle-Extendable Graphs},
  author = {Pause, Kapil R Shenvi},
  year = 2022,
  school = {Chennai Mathematical Institute}
}

@article{PengW2019Ear,
  title = {Ear Decomposition and Induced Even Cycles},
  author = {Peng, Dongmei and Wang, Xiumei},
  year = 2019,
  month = jul,
  journal = {Discrete Applied Mathematics},
  volume = {264},
  pages = {161--166},
  issn = {0166218X},
  doi = {10.1016/j.dam.2019.01.005},
  urldate = {2023-05-04},
  langid = {english}
}

@article{Plummer1980$n$extendable,
  title = {On $n$-extendable graphs},
  author = {Plummer, Michael D},
  year = 1980,
  journal = {Discrete Mathematics},
  volume = {31},
  number = {2},
  pages = {201-210},
  issn = {0012365X},
  doi = {10.1016/0012-365X(80)90037-0},
  urldate = {2023-04-30},
  langid = {english}
}

@inproceedings{Plummer1986Matching,
  title = {Matching {{Extension}} in {{Bipartite Graphs}}},
  booktitle = {Proceedings of the 17th {{Southeastern Conference}} on {{Combinatorics}}, {{Graph Theory}} and {{Computing}}},
  author = {Plummer, Michael D},
  year = 1986,
  pages = {245--258},
  address = {Winnipeg}
}

@article{RobertsonST1999Permanents,
  title = {Permanents, {{Pfaffian Orientations}}, and {{Even Directed Circuits}}},
  author = {Robertson, Neil and Seymour, Paul D and Thomas, Robin},
  year = 1999,
  month = nov,
  journal = {The Annals of Mathematics},
  volume = {150},
  number = {3},
  eprint = {121059},
  eprinttype = {jstor},
  pages = {929--975},
  issn = {0003486X},
  doi = {10.2307/121059},
  urldate = {2023-05-01}
}

@article{Szigeti1998Two,
  title = {The {{Two Ear Theorem}} on {{Matching-Covered Graphs}}},
  author = {Szigeti, Zolt{\'a}n},
  year = 1998,
  month = sep,
  journal = {Journal of Combinatorial Theory, Series B},
  volume = {74},
  number = {1},
  pages = {104--109},
  issn = {00958956},
  doi = {10.1006/jctb.1998.1824},
  urldate = {2025-06-18},
  copyright = {https://www.elsevier.com/tdm/userlicense/1.0/},
  langid = {english}
}

@article{Valiant1979Complexity,
  title = {The {{Complexity}} of {{Enumeration}} and {{Reliability Problems}}},
  author = {Valiant, Leslie G},
  year = 1979,
  month = aug,
  journal = {SIAM Journal on Computing},
  volume = {8},
  number = {3},
  pages = {410--421},
  issn = {0097-5397, 1095-7111},
  doi = {10.1137/0208032},
  urldate = {2023-04-30},
  langid = {english}
}

@article{VaziraniY1989Pfaffian,
  title = {Pfaffian Orientations, 0--1 Permanents, and Even Cycles in Directed Graphs},
  author = {Vazirani, Vijay V and Yannakakis, Milhalis},
  year = 1989,
  month = oct,
  journal = {Discrete Applied Mathematics},
  volume = {25},
  number = {1-2},
  pages = {179--190},
  issn = {0166218X},
  doi = {10.1016/0166-218X(89)90053-X},
  urldate = {2023-05-03},
  langid = {english}
}

@article{WangLCLSL2013Characterization,
  title = {A Characterization of {{PM-compact}} Bipartite and near-Bipartite Graphs},
  author = {Wang, Xiumei and Lin, Yixun and Carvalho, Marcelo H and Lucchesi, Cl{\'a}udio L. and Sanjith, G. and Little, C.H.C.},
  year = 2013,
  month = mar,
  journal = {Discrete Mathematics},
  volume = {313},
  number = {6},
  pages = {772--783},
  issn = {0012365X},
  doi = {10.1016/j.disc.2012.12.015},
  urldate = {2025-06-18},
  copyright = {https://www.elsevier.com/tdm/userlicense/1.0/},
  langid = {english}
}

@article{WangSLC2014Characterization,
  title = {A Characterization of {{PM-compact}} Claw-Free Cubic Graphs},
  author = {Wang, Xiumei and Shang, Weiping and Lin, Yixun and Carvalho, Marcelo H.},
  year = 2014,
  month = jun,
  journal = {Discrete Mathematics, Algorithms and Applications},
  volume = {06},
  number = {02},
  pages = {1450025},
  issn = {1793-8309, 1793-8317},
  doi = {10.1142/S1793830914500256},
  urldate = {2025-06-18},
  langid = {english}
}

@article{WangYL2015Characterization,
  title = {A Characterization of {{PM-compact Hamiltonian}} Bipartite Graphs},
  author = {Wang, Xiu-mei and Yuan, Jin-jiang and Lin, Yi-xun},
  year = 2015,
  month = jun,
  journal = {Acta Mathematicae Applicatae Sinica, English Series},
  volume = {31},
  number = {2},
  pages = {313--324},
  issn = {0168-9673, 1618-3932},
  doi = {10.1007/s10255-015-0475-3},
  urldate = {2023-05-01},
  langid = {english}
}

@article{WangZZ2018Characterization,
  title = {A Characterization of Cycle-Forced Bipartite Graphs},
  author = {Wang, Xiumei and Zhang, Yipei and Zhou, Ju},
  year = 2018,
  month = sep,
  journal = {Discrete Mathematics},
  volume = {341},
  number = {9},
  pages = {2639--2645},
  issn = {0012365X},
  doi = {10.1016/j.disc.2018.06.017},
  urldate = {2023-05-01},
  langid = {english}
}

@article{Whitney1932Nonseparable,
  title = {Non-Separable and Planar Graphs},
  author = {Whitney, Hassler},
  year = 1932,
  journal = {Transactions of the American Mathematical Society},
  volume = {34},
  number = {2},
  pages = {339--362},
  issn = {0002-9947, 1088-6850},
  doi = {10.1090/S0002-9947-1932-1501641-2},
  urldate = {2023-12-19},
  langid = {english}
}

@article{Whitney19332Isomorphic,
  title = {2-{{Isomorphic Graphs}}},
  author = {Whitney, Hassler},
  year = 1933,
  journal = {American Journal of Mathematics},
  volume = {55},
  number = {1},
  eprint = {2371127},
  eprinttype = {jstor},
  pages = {245--254},
  issn = {00029327},
  doi = {10.2307/2371127},
  urldate = {2023-12-19}
}

@phdthesis{Wiederrecht2022Matching,
  type = {Doctoral {{Thesis}}},
  title = {Matching Minors in Bipartite Graphs},
  author = {Wiederrecht, Sebastian},
  year = 2022,
  address = {Berlin},
  urldate = {2023-05-01},
  copyright = {Creative Commons Attribution 4.0 International},
  school = {Technische Universit\"at Berlin}
}

@article{ZhangL2012Computing,
  title = {Computing the Permanental Polynomials of Bipartite Graphs by {{Pfaffian}} Orientation},
  author = {Zhang, Heping and Li, Wei},
  year = 2012,
  month = sep,
  journal = {Discrete Applied Mathematics},
  volume = {160},
  number = {13-14},
  pages = {2069--2074},
  issn = {0166218X},
  doi = {10.1016/j.dam.2012.04.007},
  urldate = {2023-05-01},
  langid = {english}
}

@article{ZhangW2022Characterizations,
  title = {Characterizations of {{Cycle-Forced}} 2-{{Connected Claw-Free Cubic Graphs}}},
  author = {Zhang, Yi-ran and Wang, Xiu-mei},
  year = 2022,
  journal = {Chinese Quarterly Journal of Mathematics},
  volume = {37},
  number = {4},
  pages = {432--440},
  doi = {10.13371/j.cnki.chin.q.j.m.2022.04.012}
}

@article{ZhangWY2020Even,
  title = {Even Cycles and Perfect Matchings in Claw-Free Plane Graphs},
  author = {Zhang, Shanshan and Wang, Xiumei and Yuan, Jinjiang},
  year = 2020,
  journal = {Discrete Mathematics \& Theoretical Computer Science},
  volume = {22},
  number = {4},
  pages = {\#6},
  doi = {10.23638/DMTCS-22-4-6}
}

@article{ZhangWY2022PMcompact,
  title = {{{PM-compact Graphs}} and {{Vertex-deleted Subgraphs}}},
  author = {Zhang, Yi-pei and Wang, Xiu-mei and Yuan, Jin-jiang},
  year = 2022,
  month = oct,
  journal = {Acta Mathematicae Applicatae Sinica, English Series},
  volume = {38},
  number = {4},
  pages = {955--965},
  issn = {0168-9673, 1618-3932},
  doi = {10.1007/s10255-022-1018-3},
  urldate = {2025-06-18},
  langid = {english}
}

@article{ZhangWYNC2022Cyclenice,
  title = {On Cycle-Nice Claw-Free Graphs},
  author = {Zhang, Shanshan and Wang, Xiumei and Yuan, Jinjiang and Ng, C.T. and Cheng, T.C.E.},
  year = 2022,
  month = jul,
  journal = {Discrete Mathematics},
  volume = {345},
  number = {7},
  pages = {112876},
  issn = {0012365X},
  doi = {10.1016/j.disc.2022.112876},
  urldate = {2023-05-01},
  langid = {english}
}

@article{ZhouLFD2022Note,
  title = {A Note on {{PM-compact}} \$ {{K}}\_4 \$-Free Bricks},
  author = {Zhou, Jinqiu and Li, Qunfang and {Faculty of Science, Jiangxi University of Science and Technology, Ganzhou 341000, China} and {Department of Mathematics, Ganzhou Teachers College, Ganzhou 341000, China}},
  year = 2022,
  journal = {AIMS Mathematics},
  volume = {7},
  number = {3},
  pages = {3648--3652},
  issn = {2473-6988},
  doi = {10.3934/math.2022201},
  urldate = {2025-06-18}
}

\end{document}